\documentclass[11pt,a4paper,reqno]{amsart}

\usepackage{amsthm,amsfonts,amssymb,amsmath}
\usepackage{enumerate}
\usepackage{color}
\usepackage{geometry}[hmargin=2.3cm, vmargin=3cm]
\usepackage{xcolor}
\usepackage{hyperref}

\newtheorem{thm}{Theorem}[section]
\newtheorem{prp}[thm]{Proposition}
\newtheorem{cor}[thm]{Corollary}
\newtheorem{lm}[thm]{Lemma}
\newtheorem{df}[thm]{Definition}
\newtheorem{rk}[thm]{Remark}

\newcommand{\m}[1]{\mathbb{#1}}

\newcommand{\q}[1]{\mathcal{#1}}

\newcommand{\wht}[1]{\widetilde{#1}}

\newcommand{\ep}{\varepsilon}

\newcommand{\wb}[1]{\overline{#1}}

\newcommand{\alp}{\alpha}

\newcommand{\qq}{2^*}
\newcommand{\R}{\mathbb{R}}
\newcommand{\rr}{\langle x\rangle}
\renewcommand{\tilde}{\widetilde}

\numberwithin{equation}{section}

\date{\today}
\begin{document}
	\title{Einstein constraint equations for Kaluza-Klein spacetimes}

	\begin{abstract}
		The aim of this article is to construct initial data for the Einstein equations on manifolds of the form $\m R^{n+1}\times \m T^m$, which are asymptotically flat at infinity, without assuming any symmetry condition in the compact direction. We use the conformal method to reduce the constraint equations to a system of elliptic equation and work in the near CMC (constant mean curvature) regime. The main new feature of our paper is the  introduction of new weighted Sobolev spaces, adapted to the inversion of the Laplacian on product manifolds.
	\end{abstract}

\author{C\'ecile Huneau }
\address{\'Ecole Polytechnique \& CNRS}
\email{cecile.huneau@polytechnique.edu}

\author{Caterina V\^alcu}
\address{Universit\'e Sorbonne Paris Nord}
\email{valcu@math.univ-paris13.fr}

\maketitle
	\section{Introduction}
	\par In this paper, we study the constraint equations for Einstein equations on manifolds of the form $\m R^{n+1}\times \m T^m$.
	Spacetimes with compact directions were introduced almost a century ago by Theodor Kaluza and Oskar Klein \cite{Kal21,Kle26} as an early attempt of unifying electromagnetism and general relativity in a simple, elegant way. They showed that Einstein vacuum equations on $\m R^{4+1}$ with a $U(1)$ symmetry group reduce to Einstein-Maxwell-Scalar field equations on $\m R^{3+1}$. In 1968, Kerner \cite{kern} extended this study to other compact isometry groups, obtaining in this way Einstein Yang-Mills equations (see also Appendix 8 in \cite{livrecb}). Spacetimes with compact directions now play an important role in supergravity and string theory (\cite{CanHorStrWit85}). The symmetry in the compact directions present in Kaluza-Klein theory at the beginning is relaxed in more modern theories (see \cite{kk} and \cite{duff}). The presence of a non symmetric compact direction $\m T^m$ yields the creation of modes (sometimes called Kaluza-Klein tower of particles).

	\par {Einstein equations on product manifold have been the object of recent mathematical study. In \cite{Wya18}, Zoe Wyatt proves the non-linear stability under perturbations of the non-compact directions of the product of Minkowski spacetime with a $m$-dimensional torus. A parallel result for the cosmological Kaluza-Klein spacetimes, where $\R^{3+1}$ is replaced with the $4$-dimensional Milne spacetime has been proved \cite{BraFajKro19}. Finally, stability has been proven for the product of higher-dimensional Minkowski spacetimes with a compact, Ricci-flat $K$ that admits a spin structure and a nonzero parallel spinor \cite{AndBluWyaYau20}. In all these papers, the initial data are assumed to solve the constraint equations, but the issue of the existence and behaviour of the initial data is not addressed. Additionally, we would like to mention the paper of Xianzhe Dai on the positive mass theorem for manifolds including those of Kaluza-Klein type \cite{Dai04}. As an example of the analysis of initial data on manifolds with other interesting asymptotics, we cite the work of Piotr T. Chru\'sciel, Rafe Mazzeo and Samuel Pocchiola, who work on cylindrical ends. \cite{chrmaz14, chrmazpoc13}.}
	
	\medskip
	\subsection{Preliminaries}
	\par The constraint equations on a Riemannian manifold $(M,\hat{g})$ establish the relationship between the metric $\hat{g}$ and the extrinsic curvature $\hat{K}$, which describes the embedding of $M$ in a surrounding spacetime. For our model, we also accept the presence of a Klein-Gordon field \footnote{{We believe that our study can be generalized to other matter fields, in the same way that it was previously done for asymptotically Euclidean manifold. However, this is not the main focus of the present paper.}}. The resulting system takes the form
	\begin{align*}
		R_{\hat{g}}+(tr_{\hat{g}} \hat{K})^2-||\hat{K}||_{\hat{g}}^2& =  2\rho, \\
		div_{\hat{g}} \hat{K}-\nabla_{\hat{g}}tr_{\hat{g}}\hat{K}& =  J,
	\end{align*}
	where $\rho$ is the energy density of the scalar field on $M$ and $J$ is the momentum density. The conformal method is a natural way to proceed when analyzing the constraint equations \cite{Lic44, ChoYor80}. Given a conformal change $\hat{g}=\varphi^{\qq-2}g$, $\varphi>0$, and the corresponding decomposition of $\hat{K}=\varphi^{-2}(\mathcal{L}_g W+U)+\frac{2}{n}\varphi^{q-2}g$, the constraint equations take the form
	
	\begin{equation}\label{sfce}
		\begin{split}
			\frac{4(d-1)}{d-2}\Delta_g \varphi + \left(R_g -\vert \nabla\psi \vert^2_g\right) \varphi &= \left(V(\psi)-\frac{d-1}{d}\tau^2\right) \varphi^{\qq-1}+\frac{\vert U +\mathcal{L}_g W\vert_g^2+|\pi|_g^2}{\varphi^{\qq+1}}\\
			\overrightarrow{\Delta}_g W &= \frac{d-1}{d}\varphi^{\qq} d\tau+\pi d\psi.
		\end{split}
	\end{equation}
	
	Here, $\tau=tr_{\hat{g}}\hat{K}$ is the mean curvature of the spacelike hypersurface $M$, $U$ is a symmetric $2$-tensor, $\psi$ and $\pi$ are scalar functions corresponding to the scalar field and its time derivative respectively, $V$ is a smooth scalar function representing the potential of the scalar field, and the positive scalar function  $u$ and vector field  $W$ are unknown. An important point to signal is that $d$ is the dimension of the entire manifold and $2^*=\frac{2d}{d-2}$ is the critical Sobolev exponent corresponding to dimension $d$. We have also denoted the covariant derivative with respect to $g$ by $\nabla_g$, the conformal Lie derivative acting on vector fields $W$ in $M$ as 
	\begin{equation*}
		(\mathcal{L}_g W)_{\mu\nu}=\nabla_{g\,\mu} W_\nu+\nabla_{g\,\nu} W_\mu-\frac{2}{d}div_g W g_{ij},\quad\forall\, \mu,\nu\in\overline{1,d}
	\end{equation*}
	and the conformal Laplace operator $\overrightarrow{\Delta_g}$ takes the form
	\begin{equation*}
		(\overrightarrow{\Delta}_g W)^\nu=-\nabla_{g\,\mu}\left( \nabla_g^\mu W^\nu+\nabla_g^\nu W^\mu -\frac{2}{d}g^{\mu\nu}\nabla_{g\, k} W^k \right),\quad \forall\,\mu,\nu\in\overline{1,d}.
	\end{equation*}
	\par Throughout this work, the Laplace-Beltrami operators are defined such that they have positive eigenvalues, and thus negative sign.
	\medskip
	\par We consider the product manifold $M= \m R^n\times \m T^m$ of dimension $d=n+m$. Let $(x^i)_{i\in\overline{1,n}}$ be the coordinates of $\R^n$ and $(\theta^j)_{j\in\overline{1,m}}$ that of $\m T^m$. We call $(M,\zeta)$ the \textit{flat product manifold}, where the metric $\zeta$ is defined as\footnote{{Our study also directly applies to the metrics $\zeta_\alpha=dx^2+ \alpha_1(d\theta^1)^2+..+ \alpha_m(d\theta_m)^2$, with $\alpha_i$ positive constants. Note that these metrics are not isometric with each other.}}
	$$\zeta=dx^2+  d\theta^2.$$
	The corresponding Laplacian takes the form $\Delta_\zeta = -\sum_{i=1}^n \partial^{2}_{x_i} - \sum_{j=1}^m\partial_{\theta_j}^2$.
	\medskip
	\par The behaviour of the Laplacian on $\R^{n}\times \m T^m$ is different from that of the same operator on the Euclidean space. First of all, when we speak of the behaviour at infinity of a function $u$ on $\R^n\times \m T^m$, we look at $u(x,\theta)\in \R^n\times \m T^m$ for $|x|\to\infty$ and for all $\theta\in\m T^m$. So while $\Delta_{\R^n}u$ decays faster than $u$ at infinity, the decay of $\Delta_{\m T^m}u$ does not necessarily follow suit. Another way to express this is that the zero and non-zero modes of $u$ behave differently under the Laplacian. In order to account for this, we define new weighted Sobolev spaces which, while reminiscent of the function spaces that usually appear in the treatment of asymptotically Euclidean manifolds \cite{cb81, Bar86}, also keeps track of the different decay rates.
	\begin{df} The weighted Sobolev space $W^p_{s,\delta,\gamma}(\R^n\times \m T^m)$ is the completion of the space $C^\infty_0(\R^n\times\m T^m)$ of compactly supported smooth functions for the norm
		$$\|u\|_{W^{p}_{s,\delta,\gamma}(\R^n\times \m T^m)}^p = \sum_{0\leq |\beta|\leq s} \int_{\m R^n} |\partial_x^\beta \bar{u}|^p \rr^{p(\delta+|\beta|)}\,dx
		+ \sum_{0\leq |\beta|\leq s} \int_{\m R^n\times \m T^m} |\partial_{x,\theta}^\beta (u-\bar{u})|^p\rr^{p\gamma}\,dxd\theta,$$
		where $\bar{u}(x)= \frac{1}{vol(\m T^m)}\int_{\m T^m} u(x,\theta)\,d\theta$ is the average on $\m T^m$ for a fixed $x$ and $\rr= \sqrt{1+|x|^2}$. Here, $1\leq p < \infty$ and $\delta,\gamma\in\R$.
	\end{df}
	\par It is useful to also introduce notations for the separate components used in the above definition. Let us denote the following weighted Sobolev norms
	\begin{equation*}
		\|f\|^p_{W^p_{s,\delta}(\R^n)}=\sum_{0\leq |\beta|\leq s} \int_{\m R^n} |\partial_x^\beta f|^p \rr^{p(\delta+|\beta|)}\,dx
	\end{equation*}
	and
	\begin{equation*}
		\|f\|^p_{\widetilde{W}^p_{s,\gamma}(\R^n\times \m T^m)}=\sum_{0\leq |\beta|\leq s} \int_{\m R^n\times \m T^m} |\partial_{x,\theta}^\beta f\vert^p\rr^{p\gamma}\,dxd\theta.
	\end{equation*}
	Similarly, we write the H\"older norms
	\begin{equation*}
		\|f\|_{\mathcal{C}^{m_1}_{\rho_1}(\R^n)}=\sum_{0\leq l\leq m_1}\sup_{\R^n}\left(\vert\partial^l_x f\vert\rr^{\rho_1+l}\right)
	\end{equation*}
	and
	\begin{equation*}
		\|f\|_{\tilde{\mathcal{C}}^{m_2}_{\rho_2}(\R^n\times \m T^m)}=\sum_{0\leq l\leq m_2}\sup_{\R^n\times \m T^m}\left(\vert\partial^l_{x,\theta} f\vert\rr^{\rho_2} \right).
	\end{equation*}

	\medskip
	\subsection{Main result}
	\par With the preliminaries out of the way, we are ready to present the main theorem of the article. It states that initial data in the vacuum setting exist in the near-CMC regime. We require the non-negativity of the scalar curvature and certain decay rates at infinity. These conditions may not optimal, but are rather convenient.
	\begin{thm}\label{thexvac} Let $g$ be a metric on $M=\m R^n\times \m T^m$, $n+m=d$, such that $g_{ij}-\zeta_{ij} \in W^p_{2,\sigma,\lambda}(M)$ and $R_g\geq 0$, where $d<p$. We assume that the asymptotic decay of the metric $g$ verifies
		\begin{equation*} 
			-\frac{n}{p}<\sigma \quad \text{ and }\quad 2- \frac{n+m}{p}<\lambda.
		\end{equation*}
		We consider the system corresponding to the vacuum constraint equations
		\begin{align}
			\label{ham}	\frac{4(d-1)}{d-2}\Delta_g \varphi + R_g \varphi & =  -\frac{d-1}{d}\tau^2\varphi^{\qq-1}+\frac{\vert \mathcal{L}_g W+ U\vert^2}{\varphi^{\qq+1}},\\
			\label{mom}	\overrightarrow{\Delta}_g W & =  \frac{d-1}{d}\varphi^{\qq}d\tau.
		\end{align}
		Here, the mean curvature $\tau\in W^p_{1,\delta+1,\gamma}$,  $U$ is a vector field in $W^p_{1,\delta+1,\gamma}$ and the scalar curvature $R_g\in W^p_{0,\delta+2,\gamma}$ is bounded and non-negative.	We assume the following conditions on their decay 
		\begin{equation*}
			\quad -\frac{n}{p} <\delta<-\frac{n}{p}-2+\min(n,2\gamma)\quad \text{ and }\quad 0<\gamma,
		\end{equation*}
		together with the coupling condition
		\begin{equation*}
			\delta-\lambda+2<\gamma<\delta+\lambda+\frac{n}{p}.
		\end{equation*}
		If the variation of the mean curvature is bounded,
		\begin{equation*}
			|| d\tau ||_{W^{p}_{0,\delta+2,\gamma}}\leq \varepsilon,
		\end{equation*}
		with $\varepsilon>0$ sufficiently small, we find solutions $(\varphi, W)$ to the vacuum constraint equations \eqref{ham}-\eqref{mom}, where $\varphi=A+u$, $u\in W^p_{2,\delta,\gamma}$, $A>0$, $\varphi>0$ and $W\in W^p_{2,\delta,\gamma}$.
	\end{thm}
	\medskip
	\par\noindent\textbf{Comments on Theorem \ref{thexvac}.}
	\begin{enumerate}
		\item We wrote our theorem on manifolds diffeomorphic to $\m R^n \times \m T^n$, with one chart and one end, but the result can be generalized to asymptotically flat manifolds with several ends, as in the asymptotically Euclidean case.
		\item It would be interesting to generalize our result to $\m R^n \times K$ with $K$ a compact manifold. The main hiccup comes from the analysis of conformal Killing fields.
		\item We do not claim that our conditions on $(\sigma,\lambda,\delta,\gamma)$ are optimal. We can, however, ensure that they hold by simply taking $\lambda$ sufficiently large.
		\item Note also that the solutions we obtain have been taken to approach a constant on the torus $\m T^m$ as $|x|$ goes to infinity: $A>0$ in the case of $\varphi$ and $0$ in the case of $W$.
		\item The condition $R_g\geq 0$ can be relaxed, as in \cite{ChoBruIseYor00}.
	\end{enumerate}

	\par An immediate generalization of this result consists of finding initial data in the case of the constraint equations in the presence of a scalar field $\psi$. The most straightforward way to adapt the previous arguments is by asking that the coefficient of the zeroth order linear term $\tilde{h}:=R_g-|\nabla\psi|^2_g$ is non-negative and that the scalar field's potential $V$ satisfies $|V(\psi)|\leq\frac{d-1}{d}\tau^2$. These standard conditions ensure that $\Delta_g+h$ is injective and that solutions of the inhomogeneous elliptic equation $\Delta_g\varphi+\tilde{h}\varphi=v$ with $v\geq 0$ are positive.
	\begin{cor}The constraint equations in the presence of a scalar field \eqref{sfce} accept solutions given the hypotheses of Theorem \ref{thexvac} if, moreover, $|\nabla\psi|^2_g\leq R_g$ and $|V(\psi)|\leq\frac{d-1}{d}\tau^2$.
	\end{cor}
	Other similar results are conceivable, for example by fixing different matter sources. The work included here, however, mainly focuses on the intricacies of handling the vacuum case as a necessary first step. 
	\medskip
	\par\noindent\textbf{The structure of the paper.} In \textbf{Section \ref{sobsec}}, we derive the embedding and multiplication properties corresponding to the weighted Sobolev spaces, which establish the foundation for the rest of the proof. \textbf{Section \ref{secflat}} contains the bulk of the proof. We study the behaviour of elliptic operators, first on the flat manifolds, and then on asymptotically flat manifolds. Special attention needs to be given to the kernel of the conformal Laplacian: we want to show there are no nontrivial conformal Killing vector fields which decay to zero at infinity. Finally, in \textbf{Section \ref{sec_cont}}, we apply the classical barrier method to solve the scalar constraint equation, and then choose to subsequently use a fixed-point argument to find solutions to the entire system. An inverse function theorem would also have been possible.
	\medskip
	\par\noindent\textbf{Acknowledgements.} The authors would like to thank Romain Giquaud for helpful discussions in the process of writing this paper. CV is indebted to Fondation Math\' ematiques Jacques Hadamard and LabEx LMH. This work was supported by a public grant as part of the Investissement d'avenir project, reference ANR-11-LABX-0056-LMH, LabEx LMH. The first author was
supported by the ANR-19-CE40-0004
	
	\section{Weighted Sobolev spaces on a product manifold}\label{sobsec}
	As previously hinted, the introduction of the weighted Sobolev spaces is essential to the resolution of elliptic systems on the product manifold $\R^n\times \m T^m$. In particular, we need to see how classical multiplication and embedding properties apply in these spaces, much like in \cite{cb81}.
	\begin{prp}\label{multemb prop}
		The following multiplication property holds
		\begin{align}\label{mult prop}
			W^p_{s_1,\delta_1,\gamma_1}(\R^n\times \m T^m)\times W^p_{s_2,\delta_2,\gamma_2}(\R^n\times \m T^m)\subset	W^p_{s,\delta,\gamma}(\R^n\times \m T^m)
		\end{align}
		when the following conditions on the weights are verified
		\begin{gather}\label{hyp mult}
			s\leq s_1, s_2,\quad s\leq s_1+s_2-\frac{n+m}{p}, \quad \delta\leq \delta_1+\delta_2+\frac{n}{p}, \quad \delta+s \leq \gamma_1+\gamma_2,\\ \quad \gamma \leq \gamma_1+\gamma_2, \quad \gamma \leq min\left(s_i,\frac{n}{p}\right)+ \delta_i+ \gamma_j,\,\quad i \neq j. \nonumber
		\end{gather}
		Additionally, we prove the embedding properties of the weighted Sobolev and Hölder spaces:
		\begin{align}\label{comp}
			1. &\quad W^p_{s,\delta,\gamma}(\R^n\times \m T^m)\subset\subset W^p_{s',\delta',\gamma'}(\R^n\times \m T^m),\quad&\quad s<s',\quad\delta<\delta',\quad\gamma<\gamma',\\
			2. &\quad W^p_{s,\delta}(\R^n)\subset \mathcal{C}^{m_1}_{\rho_1}(\R^n),\quad&\quad m_1<s-\frac{n}{p},\quad \rho_1<\delta+\frac{n}{p}, \nonumber \\
			3. &\quad \tilde{W}^p_{s,\gamma}(\R^n\times \m T^m)\subset \tilde{C}^{m_2}_{\gamma}(\R^n\times \m T^m)\quad&\quad m_2<s-\frac{n+m}{p}. \nonumber
		\end{align}
	\end{prp}
	This result can easily be generalized for $K$ an arbitrary compact set instead of $\m T^m$.
	\begin{proof}$\quad$
		\par\textit{The multiplication property \eqref{mult prop}.} Let $u\in W^p_{s_1,\delta_1,\gamma_1}(\R^n\times \m T^m)$ and $v\in W^p_{s_2,\delta_2,\gamma_2}(\R^n\times \m T^m)$. In our calculations, we can assume that $u,v \in \mathcal{C}^\infty_0(\R^n\times \m T^m)$ and obtain the desired property through a standard density argument.
		\par We begin by estimating the spherically symmetric term of the $W^{p}_{s,\delta,\gamma}$ norm of the product $uv$, 
		\begin{align*}
			\sum_{0\leq \vert\beta\vert\leq s} \int_{\R^n}\vert \partial^\beta_x \overline{uv}\vert^p \langle x \rangle^{p(\delta+|\beta|)}\,dx&=  \sum_{0\leq\vert\alpha\vert \leq\vert\beta\vert\leq s} \int_{\R^n}\vert  \overline{\partial^{\vphantom{\beta}\alpha}_x u\partial^{\beta-\alpha}_x v}\vert^p \langle x \rangle^{p(\delta+|\beta|)}\,dx \\
			& \leq  \sum_{0\leq\vert\alpha\vert \leq\vert\beta\vert\leq s} \int_{\R^n} \vert  {\partial^{\vphantom{\beta}\alpha}_x \overline{u}}\,{\partial_x^{\beta-\alpha}\overline{v}}\vert^p \langle x \rangle^{p(\delta+|\beta|)}\,dx\\
			&\quad + \sum_{0\leq\vert\alpha\vert \leq\vert\beta\vert\leq s} \int_{\R^n}\left\vert  \overline{\partial_x^{\vphantom{\beta}\alpha} (u-\bar{u})\partial_x^{\beta-\alpha}(v-\bar{v})}\right\vert^p \langle x \rangle^{p(\delta+|\beta|)}\,dx.
		\end{align*}
		Here, we used the basic identity $\overline{fg}=\bar{f}\bar{\vphantom{f}g}+\overline{(f-\bar{f})(g-\bar{g})}$ for $f,g\in\mathcal{C}^\infty_0(\R^n\times \m T^m)$. The first integral is controlled as in the paper of Choquet-Bruhat, Isenberg and York on the existence of initial data solution on asymptotically Euclidean manifolds \cite{ChoBruIseYor00}. To be precise, we obtain that
		$$\sum_{0\leq\vert\alpha\vert \leq\vert\beta\vert\leq s} \int_{\R^n} \vert  \partial^{\alpha}_x \overline{u}\,\partial_x^{\beta-\alpha}\overline{v}\vert^p \langle x \rangle^{p(\delta+\beta)}\,dx\ \lesssim \|\bar{u}\|_{W^p_{s_1,\delta_1}}\|\bar{v}\|_{W^p_{s_2,\delta_2}},$$
		if the following conditions are satisfied
		\begin{equation}\label{cond1}
			s\leq s_1,s_2, \quad s<s_1+s_2-\frac{n}{p}, \quad \delta\leq \delta_1+\delta_2+\frac{n}{p} .
		\end{equation}
		For the second integral, we use Jensen's inequality. We recall that the classical result states that for any $f$ an integrable function on $\m T^m$, 
		$$\vert\bar{f}\vert^p=\left \vert\frac{1}{{vol(\m T^m)}}\int_{\m T^m}f\,d\theta\right\vert^p\leq \frac{1}{{vol(\m T^m)}}\int_{\m T^m}\vert f\vert^p\,d\theta=\overline{\vert f\vert^p}.$$
		By combining the above with H\"older's inequality, and moreover with the fact that $\langle x\rangle^{\alpha_1}\leq C\langle x\rangle^{\alpha_2}$ for $0\leq \alpha_1\leq \alpha_2$, it follows that
		\begin{align*}
			&\int_{\R^n}\left\vert  \overline{\partial_{x}^{\alpha\vphantom{\beta}} (u-\bar{u})\partial^{\beta-\alpha}_{x}(v-\bar{v})}\right\vert^p \langle x \rangle^{p(\delta+|\beta|)} \\
			&\leq  \frac{1}{{vol(\m T^m)}}\int_{\R^n\times \m T^m}\vert \partial^\alpha_{x} (u-\bar{u})\partial_{x}^{\beta-\alpha}(v-\bar{v}) \vert^p \langle x \rangle^{p(\delta+|\beta|)} \\
			&\lesssim  \| \partial_{x}^\alpha(u-\bar{u})\langle x\rangle^\theta \|^p_{L^{pa}(\R^n\times \m T^m)}  \\
			&\quad\quad\quad\times \|\partial_{x}^{\beta-\alpha}(v-\bar{v})\langle x\rangle^\lambda \|^p_{L^{pb}(\R^n\times \m T^m)},
		\end{align*}
		with
		\begin{equation}\label{condint}
			\frac{1}{a}+\frac{1}{b}=1\quad\text{ and }\quad \delta+|\beta| \leq \theta+\lambda.
		\end{equation}
		The last condition is satisfied for all $|\beta|\leq s$ if
		\begin{equation}\label{condint2}
			\delta+s\leq \theta+\lambda.
		\end{equation}
		We apply the Sobolev embedding for ${W^{s_1,p}}(\R^n\times \m T^m)\subset {W^{\alpha,pa}}(\R^n\times \m T^m)$ with $s_1\in\R$ such that
		\begin{equation}\label{condint3}\frac{n+m}{p}-s_1<\frac{n+m}{pa}-\alpha, 
		\end{equation}
		to get
		\begin{align*}
			\|\partial_x^\alpha (u-\bar{u})\langle x \rangle ^\theta \|_{L^{pa}(\R^n\times \m T^m)} \leq \|(u-\bar{u})\langle x\rangle^\theta \|_{W^{\alpha,pa} (\R^n\times \m T^m) }&\leq \|(u-\bar{u})\langle x\rangle^\theta \|_{W^{s_1,p} (\R^n\times \m T^m) }.
		\end{align*}
		For		
		\begin{equation}\label{condint4}\theta<\gamma_1,
		\end{equation}
		we arrive at
		\begin{align*}
			\|(u-\bar{u})\langle x\rangle^\theta \|_{W^{s_1,p} (\R^n\times \m T^m) }\leq \|u-\bar{u}\|_{\tilde{W}^p_{s_1,\gamma_1}(\R^n\times \m T^m)}.
		\end{align*}
		Similarly, for the remaining term, we see that
		\begin{equation*}
			\|\partial_{x}^{\beta-\alpha}(v-\bar{v})\langle x\rangle^\lambda \|_{L^{pb}(\R^n\times \m T^m)}\leq C\|v-\bar{v}\|_{\tilde{W}^p_{s_2,\gamma_2}(\R^n\times \m T^m)},
		\end{equation*}
		where
		\begin{equation}\label{condint5}\frac{n+m}{p}-s_2<\frac{n+m}{pb}-(\beta-\alpha)
		\end{equation}
		and
		\begin{equation}\label{condint6}
			\lambda\leq \gamma_2
		\end{equation}
		Conditions \eqref{condint2}, \eqref{condint4} and \eqref{condint6} can be satisfied if and only if
		\begin{equation}
			\label{cond2} \delta+s\leq\gamma_1+\gamma_2
		\end{equation}
		and conditions \eqref{condint}, \eqref{condint3} and \eqref{condint5} can be satisfied if and only if
		\begin{equation}\label{cond3}
			s<s_1+s_2-\frac{n+m}{p}.
		\end{equation} 
		\par We look at the second term of the $W^p_{s,\delta,\gamma}(\R^n\times \m T^m)$ norm of $uv$: 
		$$\sum_{0\leq |\beta|<s}\int_{\R^n\times \m T^m}|\partial^\beta(uv-\overline{uv})|^p\langle x\rangle ^{p\gamma}\,dxd\theta.$$ Knowing that $uv=[\bar{u}-(\bar{u}-u)][\bar{v}-(\bar{v}-v)]$ and that $\overline{uv}=\bar{u}\bar{v}+\overline{(u-\bar{u})(v-\bar{v})}$, it suffices to control 
		\begin{equation*}
			\int_{\R^n\times \m T^m}\vert\partial_{x,\theta}^\beta \left( \bar{u}(v-\bar{v})+\bar{v}(u-\bar{u})-(u-\bar{u})(v-\bar{v}) +\overline{(u-\bar{u})(v-\bar{v})} \right) \vert^p\langle x\rangle^{p\gamma}\,dx d\theta.
		\end{equation*}
		Using Hölder's inequality, we estimate the term
		\begin{align*}\label{autre}
			&\sum_{0\leq |\beta|\leq s}\int_{\R^n\times \m T^m} | \partial_{x,\theta}^\beta \overline{(u-\bar{u})(v-\bar{v})}|^p\langle x\rangle^{p\gamma}\,dxd\theta\\
			&\quad\quad={vol(\m T^m)}\sum_{0\leq|\beta|\leq s}\int_{\R^n} | \partial_{x}^\beta \overline{(u-\bar{u})(v-\bar{v})}|^p\langle x\rangle^{p\gamma}\,dx\\
			&\quad\quad\lesssim\sum_{0\leq |\xi|\leq |\beta|\leq s}\int_{\R^n\times \m T^m}\vert \partial^\xi_{x} (u-\bar{u})\partial_{x}^{\beta-\xi}(v-\bar{v}) \vert^p \langle x \rangle^{p\gamma}\,dx d\theta \\
			&\quad\quad\lesssim\sum_{0\leq |\xi|\leq |\beta|\leq s} \| \partial_{x}^\xi(u-\bar{u})\langle x\rangle^\theta \|^p_{L^{pa}(\R^n\times \m T^m)}  \|\partial_{x}^{\beta-\xi}(v-\bar{v})\langle x\rangle^\lambda \|^p_{L^{pb}(\R^n\times \m T^m)}
		\end{align*}
		for 
		$$	\frac{1}{a}+\frac{1}{b}=1\quad\text{ and }\gamma\leq \theta+\lambda.$$ 
		We control both factors by $\tilde{W}^p_{s_1,\gamma_1}$ and $\tilde{W}^p_{s_2, \gamma_2}$ norms, respectively, given that 
		\begin{equation}\label{cond4}\gamma\leq \gamma_1+\gamma_2
		\end{equation}
		and	\eqref{cond3} hold. Both 
		$$\sum_{0\leq |\beta|\leq s}\int_{\R^n\times \m T^m} | \partial_{x,\theta}^\beta \bar{u}(v-\bar{v})|^p\langle x\rangle^{p\gamma}\quad\text{ and }\quad{\sum_{0\leq |\beta|\leq s}\int_{\R^n\times \m T^m} | \partial_{x,\theta}^\beta (u-\bar{u})\bar{v}|^p\langle x\rangle^{p\gamma}}$$ 
		are similarly controlled. Thus, for
		\begin{equation}\label{condint13} \gamma\leq \theta+\lambda,
		\end{equation}
		we have
		\begin{equation}\label{prodest}
		\begin{split}
			&\sum_{0\leq |\beta|\leq s}\int_{\R^n\times \m T^m} | \partial_{x,\theta}^\beta \bar{u}(v-\bar{v})|^p\langle x\rangle^{p\gamma}\\
			\lesssim&\sum_{0\leq |\xi|\leq |\beta|\leq s} \| \partial_{x}^\xi\bar{u}\langle x\rangle^\theta \|^p_{L^{pa}(\R^n)}  \|\partial_{x,\theta}^{\beta-\xi}(v-\bar{v})\langle x\rangle^\lambda \|^p_{L^{pb}(\R^n\times \m T^m)},
			\end{split}
		\end{equation}
		\begin{align*}
			&\sum_{0\leq |\beta|\leq s}\int_{\R^n\times \m T^m} | \partial_{x,\theta}^\beta \bar{v}(u-\bar{u})|^p\langle x\rangle^{p\gamma}\\
			\lesssim&\sum_{0\leq |\xi|\leq |\beta|\leq s} \| \partial_{x}^\xi\bar{v}\langle x\rangle^\theta \|^p_{L^{pa}(\R^n)}  \|\partial_{x,\theta}^{\beta-\xi}(u-\bar{u})\langle x\rangle^\lambda \|^p_{L^{pb}(\R^n\times \m T^m)},
		\end{align*}
		as well as
		\begin{align*}
		&	\sum_{0\leq |\beta|\leq s}\int_{\R^n\times \m T^m} | \partial_{x,\theta}^\beta (u-\bar{u})(v-\bar{v})|^p\langle x\rangle^{p\gamma}\\
		\lesssim &\sum_{0\leq |\xi|\leq |\beta|\leq s} \| \partial_{x,\theta}^\xi(u-\bar{u})\langle x\rangle^\theta \|^p_{L^{pa}(\R^n\times\m T^m)}  \|\partial_{x,\theta}^{\beta-\xi}(v-\bar{v})\langle x\rangle^\lambda \|^p_{L^{pb}(\R^n\times \m T^m)}.
		\end{align*}		
		We are left with estimating the right-hand side of the above inequalities. We detail here the proof for \eqref{prodest}; the two others follow in the same way. We use different Sobolev embeddings for $s_1-|\xi|>\frac{n}{p}$ and $s_1-|\xi|\leq \frac{n}{p}$. In the first case, we fix $a>1$ such that 
		\begin{equation}\label{condint14}\frac{n}{p}-(s_1-|\xi|)= \frac{n}{pa}
		\end{equation} 
		and get
		\begin{align*}
			\|\partial_x^\xi \bar{u}\langle x\rangle^\theta \|_{L^{pa}(\R^n)} & \leq  \|\partial_x^\xi \bar{u}\|_{W^p_{s_1-|\xi|,\theta}(\m R^n)}\\& \lesssim \left(\sum_{|t|\leq s_1-|\xi|}\int_{\R^n}|\partial_x^{\xi+t}\bar{u}|^p\langle x\rangle^{\theta-\xi+\xi+t}\,dx\right)^{\frac{1}{p}} \\
			& \lesssim  \|\bar{u}\|_{W^p_{s_1,\delta_1}(\R^n)}
		\end{align*}
		if 
		\begin{equation}\label{condint11}\theta\leq \delta_1+|\xi|.
		\end{equation}
	 Similarly, it also holds that
		\begin{equation}
			\|\partial_{x,\theta}^{\beta-\xi} (v-\bar{v})\langle x\rangle^\lambda\|_{L^{pb}(\R^n\times \m T^m)}\leq \|(v-\bar{v}) \|_{\tilde{W}^p_{s_2,\gamma_2}(\R^n\times \m T^m)}
		\end{equation}
		if 
		\begin{equation}\label{condint15}\frac{n+m}{p}-s_2+\beta-\xi=\frac{n+1}{pb},
		\end{equation} 
		\begin{equation}\label{condint16}\lambda\leq \gamma_2.
		\end{equation}
		Condition \eqref{condint11} becomes
		\begin{equation}\label{condint12}
			\theta \leq s_1+\delta_1,
		\end{equation}
		so \eqref{condint13} and \eqref{condint16} imply
		\begin{equation}\label{cond5}
			\gamma \leq s_1+ \delta_1+ \gamma_2.
		\end{equation}
		Moreover, \eqref{condint14} and \eqref{condint15} are covered by \eqref{cond3}.
		Let us look at the case $s_1-|\xi| >\frac{n}{p}$. We have the $L^\infty$ estimate
		$$|\partial_x^\xi \bar{u}|\leq \langle x\rangle ^{-\delta_1-\frac{n}{p}-|\xi|}\|\bar{u}\|_{W^p_{s_1,\delta_1}}.$$
		Consequently, it holds that
		$$\int_{\R^n\times \m T^m} | \partial_{x,\theta}^\beta \bar{u}(v-\bar{v})|^p\langle x\rangle^{p\gamma}\,dxd\theta
		\lesssim  \|\bar{u}\|_{W^p_{s_1,\delta_1}}\|v-\bar{v}\|_{\wht W^p_{s_2,\gamma_2}}$$
		if $\gamma-\delta_1-|\xi|-\frac{n}{p}\leq \gamma_2$. Since this must be true for all $0\leq |\xi|\leq s_1-\frac{n}{p}$ we obtain the condition
		\begin{equation}\label{cond6}\gamma\leq \gamma_2 +\delta_1+ \frac{n}{p}.
		\end{equation}
		Conditions \eqref{cond1}, \eqref{cond2}, \eqref{cond3}, \eqref{cond4} , \eqref{cond5}, \eqref{cond6} (and the symmetric cases) can be summarized in
		\begin{gather*}
			s\leq s_1, s_2,\quad s\leq s_1+s_2-\frac{n+m}{p}, \quad \delta\leq \delta_1+\delta_2+\frac{n}{p}, \quad \delta+s \leq \gamma_1+\gamma_2,\\ \quad \gamma \leq \gamma_1+\gamma_2, \quad \gamma \leq min(s_i,\frac{n}{p})+ \delta_i+ \gamma_j,\,\quad i \neq j. \nonumber
		\end{gather*}
		\bigskip
		\par \textit{The Sobolev and Hölder embeddings \eqref{comp}.} This proof is a straightforward adaptation of the argument in the paper of Choquet-Bruhat Christodoulou \cite{cb81}. We consider a sequence $\{f_n\}_{n\in\mathbb{N}}$ in the unit ball of $W^p_{s,\delta,\gamma}(\R^n\times \m T^m)$. Let $B_R$ be the open ball $B_R=\{x\in \R^n, |x|<R\}$ and let $\chi_R$ be a smooth cutoff function such that $\chi_R\equiv 1$ on $B_R\times \m T^m$ and $\chi_R\equiv 0$ on $^C(B_{2R}\times \m T^m)$. Up to a subsequence, $f_n$ converges weakly to $f\in W^p_{s,\delta,\gamma}$.  In order to obtain strong $W^p_{s',\delta',\gamma'}$ convergence, we rewrite $f_n=\chi_R f_n+(1-\chi_R) f_n$. First, we look at what happens in the interior of the manifold. The sequence $\{\chi_R f_n\}$ is bounded in the classical Sobolev space $W^{s,p}(B_{2R}\times \m T^m)$,
		\begin{equation}
			\|\chi_R f_n\|_{W^p_s(B_{2R}\times \m T^m)}\leq C_R\|f_n \|_{W^p_{s,\delta,\gamma}(B_{2R}\times \m T^m)}. 
		\end{equation}
		By Rellich's compactness theorem, there exists a subsequence  of $\{\chi_R f_n\}$which  converges strongly in $W^{s',p}$ to $f$ in $B_R\times \m T^m$. In order to finish to proof, it remains to study the behaviour of $\{f_n\}_{n\in\mathbb{N}}$ at infinity. We see that
		
		\begin{align*}
			\|f-f_n\|_{W^p_{s',\delta',\gamma'}(\R^n\times \m T^m)} & \leq 
			\left\{ \sum_{\beta=0}^{s'} \int_{B_R\times \m T^m} \langle x\rangle^{p(\delta'+m)}|\partial_{x,\theta}^\beta(f-f_n)|^p \,dxd\theta\right\}^{1/p}\\
			&+\{ \int_{B_R\times \m T^m} \langle x\rangle^{p\gamma}|\partial_{x,\theta}^\beta (f-f_n -\overline{f-f_n})|^p\,dxd\theta \}^{1/p}.\\
			&+R^{\delta'-\delta}\left\{ \sum_{\beta=0}^{s'}\int_{^C B_R} \langle x\rangle^{p(\delta+m)} |\partial_{x}^k (\overline{f-f_n})|^p\,dxd\theta \right\}^{1/p}\\
			&+R^{\gamma'-\gamma}\{ \int_{^C(B_R\times \m T^m)} \langle x\rangle^{p\gamma}|\partial_{x,\theta}^\beta (f-f_n -\overline{f-f_n})|^p\,dxd\theta \}^{1/p}.
		\end{align*}
		Since $\delta'<\delta$ and $\gamma'<\gamma$, we can find $R$ such that $\sup(R^{\delta'-\delta},R^{\gamma'-\gamma})\sup_{n}\|f_n\|_{W^p_{s,\delta,\gamma}} \leq \ep$, and then use the convergence on $B_R$ to conclude.
	\end{proof}
	It is important that the multiplication is ``stable'', in a specific sense, {for the non-linearity in our equations to behave well}.
	\begin{cor}\label{mult cor 1} The following restrictions ensure that the multiplications we need are stable. If
		\begin{equation}\label{stable hyp gen}
			\gamma \geq 0,\quad 2> \frac{n+m}{p},\quad \delta\geq -\frac{n}{p}\quad \text{ and }\quad 2\gamma \geq \delta +2,
		\end{equation}
		then we have the desired properties
		\par  $W^p_{2,\delta,\gamma}\times W^p_{2,\delta,\gamma}\subset W^p_{2,\delta,\gamma}$, 
		\par  $W^p_{1,\delta+1,\gamma}\times W^p_{1,\delta+1,\gamma} \subset W^p_{0,\delta+2,\gamma}$,
		\par  $ W^p_{0,\delta+2,\gamma}\times W^p_{2,\delta,\gamma} \subset W^p_{0,\delta+2,\gamma}$.
	\end{cor}
	In the rest of the paper, we always assume the conditions of $\eqref{stable hyp gen}$ to hold true. While this does not ensure that the decay we obtain is necessarily optimal, it has the advantage of considerably simplifying calculations and notations. 
	
	\section{Elliptic theory on the  product manifold}\label{secflat}
	
	\par In order to construct solutions for the constraint equations, we must first understand the behaviour of the Laplace-Beltrami operator $\Delta_g$ and of the conformal Laplacian $\overrightarrow{\Delta}_g$ on the weighted Sobolev spaces we defined in $M=\m R^n \times \m T^m$. We begin by looking at a wider class of elliptic operators.
	\begin{df}
		{	Let $P$ be a homogeneous self-adjoint second order elliptic operator with constant coefficients, acting on scalar functions or on vector fields of $M=\m R^n \times \m T^n$. We write $$P (u)=\sum_{|\alpha|=2}B^{(\alpha)}\nabla^\alpha u,$$
			where the $B^{(\alpha)}$ are either scalars or matrices.
			For a vector operator, the ellipticity condition means that for all $\xi \in \m R^{n+m}$ with $\xi \neq 0$, $v\mapsto\sum_{|\alpha|=2}\xi^\alpha B^{(\alpha)}(v)$ is an ismorphism on $\m R^{n+m}$. Let $s\geq 2$. We say that a second order elliptic operator $L$ is asymptotic to $P$ in $W^p_{s,\sigma,\lambda}$ if 
			$$L(u)= \sum_{|\alpha|\leq 2}a^{(\alpha)}\nabla^\alpha u,$$
			with
			\begin{itemize}
				\item for $|\alpha|=l$, $a^{(\alpha)}-B^{(\alpha)} \in W^p_{s,\sigma,\lambda}$ ,
				\item for $ |\alpha|<l$, 
				$a^{(\alpha)} \in W^p_{s+|l|-2,\sigma+|l|,\lambda }$.
			\end{itemize}
			In the same way, if $P_0$ is an operator in $\m R^n$, we define the notion of being asymptotic to $P_0$ in $W^p_{s,\delta}$.}
	\end{df}
	These operators verify a series of properties described below. 
	\begin{thm}{\label{thlin}Let $P$ be a second order homogeneous elliptic self-adjoint operator, acting on functions or vector fields of $\m R^n\times \m T^m$ with constant coefficients, and $L$ be a second order elliptic operator asymptotic to $P$ in $W^p_{2,\sigma,\lambda}$ with $ \frac{n+m}{p}<2$, $-\frac{n}{p}<\sigma$ and $\lambda>0$. Let $\delta,\gamma$ be such that
			$$\delta-\lambda+2<\gamma <\delta+\lambda+\frac{n}{p}.$$
			\begin{itemize}
				\item $L:W^p_{2,\delta,\gamma}\to W^p_{0,\delta+2,\gamma}$  is a continuous map,
				\item If $-\delta-\frac{n}{p} \notin \m N$ then  $L : W^p_{2,\delta,\gamma}\to W^p_{0,\delta+2,\gamma}$ has finite dimensional kernel and closed range, and there exist $C$ and $R>0$ such that for all $u \in  W^p_{2,\delta,\gamma}$,
				\begin{equation}\label{estimp}
					\|u\|_{{W}^p_{2,\delta,\gamma}}\leq C\left(\|Lu\|_{{W}^p_{0,\delta+2,\gamma}}+\|\wb u\|_{L^p(B_R)}\right). 
				\end{equation}
		\end{itemize}}
	\end{thm}
	In the particular case of $\Delta_g$ or $\overrightarrow{\Delta}_g$, we can prove that given sufficient decay, the operators are isomorphisms.
	\begin{thm}\label{thinv}{
			Let $g\in W^p_{2,\sigma,\lambda}$ with  $ \frac{n+m}{p}<2$, $-\frac{n}{p}<\sigma$ and $\lambda>0$. Let $\delta,\gamma$ be such that
			$$\delta-\lambda+2<\gamma<\delta+\lambda+\frac{n}{p},$$
			and also
			$$0<\gamma, \quad -\frac{n}{p} <\delta<-\frac{n}{p}+n-2, \quad 2-\frac{n+m}{p}<\lambda.$$
			Let $h\in W^p_{0,\delta+2,\gamma}$ be a non negative bounded function.
			Then $\Delta_g+h$ and $\overrightarrow{\Delta_g}$ are isomorphisms $W^p_{2,\delta,\gamma}(\m R^n\times \m T^m)\to W^p_{0,\delta+2,\gamma}(\m R^n\times \m T^m).$} Moreover, we have
		$$\|u\|_{W^p_{2,\delta,\gamma}}\leq C \|L u \|_{W^p_{0,\delta+2,\gamma}},$$
		with $L$ being either $\Delta_g+h$ or $\overrightarrow{\Delta_g}$.
	\end{thm}	
	{Theorems \ref{thlin} and \ref{thinv} are the analogue of the following classical theorem in asymptotically
		Euclidean manifolds, which we cite here for reference in a condensed form, (see Appendix 2 of \cite{livrecb} and Theorem 1.10 in \cite{Bar86}).}
	\begin{thm} \label{thcb}{Let $P_0$ be a second order homogeneous elliptic operator acting on functions or vector fields of $\m R^n$ with constant coefficients and $L_0$ be a second order elliptic operator asymptotic to $P_0$ in $W^p_{2,\sigma}$ with $ \frac{n}{p}<2$ and $-\frac{n}{p}<\sigma$.
			\begin{itemize}
				\item For any $\delta$, we have $L_0 : W^p_{2,\delta}\to W^p_{0,\delta+2}$.
				\item If $-\delta-\frac{n}{p} \notin \m N$, then $L_0 : W^p_{2,\delta}\to W^p_{0,\delta+2}$ has finite dimensional kernel and closed range.
				\item There exist $C$ and $R>0$ such that for all $u \in  W^p_{2,\delta}$,
				\begin{equation}
					\|u\|_{{W}^p_{2,\delta}}\leq C\left(\|L_0u\|_{{W}^p_{0,\delta+2}}+\|u\|_{L^p(B_R)}\right). 
				\end{equation}
			\end{itemize}
			Moreover, if $g\in W^p_{2,\sigma}$  and $-\frac{n}{p}<\delta<-\frac{n}{p}+n-2$, then $\Delta_g$ and $\overrightarrow{\Delta_g}$ are isomorphisms acting from $W^p_{2,\delta}$ into $W^p_{0,\delta+2}$.}
	\end{thm}
	
	{To prove Theorems \ref{thlin} and \ref{thinv}, we need to start with $P$ the homogeneous second order elliptic operator with constant coefficients. In Section \ref{secPnonzero} we study the non-zero modes of $P$. The behaviour of the zero mode is described by Theorem \ref{thcb}. In Section \ref{secP}, we show that Theorem \ref{thlin} holds in the particular case of $L=P$. We finish the proof of Theorem \ref{thlin} in Section \ref{secthlin}. This passage from $P$ to $L$ is very similar to the asymptotically Euclidean case. Finally, in Section \ref{injsec}, we study the injectivity of $\Delta_g$ and $\overrightarrow{\Delta_g}$, and conclude the proof of Theorem \ref{thinv}.}

	\subsection{Analysis of $P$ on non zero modes}\label{secPnonzero}
	{Let  $P$ be a second order elliptic homogeneous self-adjoint operator with constant coefficients on $\R^n \times \m T^m$,
		\begin{equation}\label{L_0}
			P=-\sum_{\substack{\alpha\in\m Z^d,\\ |\alpha|=2}}B^{(\alpha)}\nabla^\alpha_{\zeta},
		\end{equation}
		where $B^{(\alpha)}$ are scalars or matrices with constant coefficients. We consider the equation $P u =f$. Since $P$ has constant coefficients, this equation implies
		$$P(u-\bar{u})=f-\bar{f}.$$
		The aim of this section is to prove the following proposition.}
	
	
	\begin{prp}\label{lms1}
		Let $\gamma,\delta \in \m R$, and  $u \in W^p_{2,\delta,\gamma}(\R^n\times \m T^m), f\in W^p_{0,\delta+2,\gamma}(\R^n\times \m T^m)$ such that $P (u-\wb u)=f-\wb f$. Then
		$$\|u-\wb u\|_{\tilde{W}^p_{2,\gamma}}\lesssim \|f-\wb f\|_{\tilde{W}^p_{0,\gamma}}.$$
		In particular, $P (u-\wb u)=0$ implies $u-\wb u=0$.
	\end{prp}
	
	We begin by introducing some useful preliminaries and notations.
	\par{	Let us take the Fourier transform in variables $x,\theta$ of the equation $P (u-\wb u)=f-\wb f$. We can write
		$$P(\xi,k) \q F_{x,\theta}(u-\wb u)(\xi,k)=\q F_{x,\theta}(f-\wb f)(\xi,k),$$
		where we denote by $(\xi,k)$ the Fourier variable associated to $(x,\theta)$ (note that $k\in \m Z^m$), and $P(\xi,k)$ is either a scalar function or a $d\times d$ matrix, whose coefficients are homogeneous polynomials of order $2$ in $\xi$ and $k$.
		The assumption that $P$ is elliptic means that $P(\xi,k)$ is invertible for $(\xi,k)\neq (0,0)$. By homogeneity, this implies that, for $(\xi,k)\neq (0,0)$,
		\begin{equation}\label{condel}\|P^{-1}(\xi,k)\|\leq \frac{C}{|\xi|^2+k^2}.
		\end{equation}
		Since $\q F_{x,\theta}(u-\wb u)(\xi,k)$ and $\q F_{x,\theta}(f-\wb f)(\xi,k)$ are supported away from $k=0$, we can write
		$$ \q F_{x,\theta}(u-\wb u)(\xi,k)=P^{-1}(\xi,k)\q F_{x,\theta}(f-\wb f)(\xi,k),$$
		and
		\begin{equation}\label{solu}u-\wb u = \q F^{-1}_{x,\theta}(P^{-1}(\xi,k)\q F_{x,\theta}(f-\wb f)(\xi,k)).
		\end{equation}
		By taking order l derivatives, we obtain 
		\begin{equation}\label{hj}
			\nabla_{x,\theta}^l (u-\wb u)  =  \q F^{-1}_x\left(\sum_{\substack{k\in \m Z^{m*}}}p_l(\xi,k) P(\xi,k)^{-1}\q F_{x,\theta}(f-\wb f)(\xi,k) e^{ik\cdot\theta}\right),
		\end{equation}
		where we denote by $p_l$ a homogeneous polynomial or rational fraction of order $l$.
		\par Our goal is to study the behaviour of derivatives of $u-\bar{u}$ of order zero, one and two. Given \eqref{hj}, we sometimes choose to concentrate our attention only on estimates of order two, the more delicate case, when the other estimates follow trivially using the same reasoning. This is to avoid adding unnecessarily heavy notation.
		Consequently, we write 
		\begin{equation}\label{h def}
			\nabla_{x,\theta}^2(u-\wb u)= h*(f-\wb f),
		\end{equation}	
		where the kernel $h$ can be expressed as
		\begin{equation}\label{defh}
			h(x,\theta)=F^{-1}_x \left(\sum_{\substack{k\in \m Z^{m*}}}p_2(\xi,k)P(\xi,k)^{-1}e^{ik\cdot\theta}\right).
		\end{equation}
		Let us denote the corresponding symbol as
		\begin{equation*}
			a(\xi,k)= p_2(\xi,k)P(\xi,k)^{-1}.
		\end{equation*}
		\par With the preliminaries out of the way, we can focus on the proof of Proposition \ref{lms1}. The desired estimate follows from the next two lemmas. The first is simply the corresponding $L^2$ estimate. Note that we are interested in both positive and negative weight values $\delta$ and $\gamma$. This becomes useful when studying the behaviour of the adjoint operator $P^*$ later in the paper.
		\begin{lm}\label{lm2}
			Let $\gamma,\delta \in \m R$ and  $u \in W^2_{2,\delta,\gamma}(\R^n\times\m T^m), f\in W^2_{0,\delta+2,\gamma}(\R^n\times \m T^m)$ such that $P (u-\wb u)=f-\wb f$. Then we have
			$$\|u-\wb u\|_{\tilde{W}^2_{2,\gamma}}\lesssim \|f-\wb f\|_{\tilde{W}^2_{0,\gamma}}.$$
		\end{lm}
		\begin{proof}
			From \eqref{condel}, \eqref{solu} and its equivalent for derivatives of $u-\bar{u}$ of order one and two, we obtain, thanks to the Plancherel formula, that
			\begin{equation}\label{est0}
				\|u-\wb u\|_{\tilde{W}^2_{2,0}}\lesssim \|f-\wb f\|_{\tilde{W}^2_{0,0}}.
			\end{equation}
			Note that $\tilde{W}^2_{2,0}=H^2$ and $\tilde{W}^2_{0,0}=L^2$. We can prove Lemma \ref{lm2} iteratively for $\gamma>0$ in the following way. 
			First, we note that
			\begin{equation}\label{eqgamma}P \left((u-\wb u)\langle x\rangle^{\gamma}\right)=(f-\wb f)\langle x\rangle^\gamma+g,
			\end{equation}
			with 
			$$\|g\|_{\tilde W^2_{0,0}}\lesssim \|u-\wb u \|_{\tilde{W}^2_{1,\gamma-1}}+\|u-\wb u \|_{\tilde{W}^2_{0,\gamma-2}}.$$
			Consequently, if we take $0<\gamma\leq 1$ we obtain
			\begin{align*}
				\|	(u-\wb u)\langle x\rangle^{\gamma}\|_{\tilde W^2_{2,0}}&\lesssim
				\|(f-\wb f)\langle x\rangle^{\gamma}\|_{\tilde{W}^2_{0,0}}+\|g\|_{\tilde{W}^2_{0,0}}\\
				& \lesssim \|f-\wb f\|_{\tilde{W}^2_{0,\gamma}}+\|u-\wb u \|_{\tilde{W}^2_{1,0}}\\
				&\lesssim \|f-\wb f\|_{\tilde{W}^2_{0,\gamma}}
			\end{align*}
			thanks to \eqref{est0}. Moreover,
			$$\|u-\wb u\|_{\tilde{W}^2_{2,\gamma}} \lesssim \|	(u-\wb u)\langle x\rangle^{\gamma}\|_{\tilde W^2_{2,0}}
			+ \|u-\wb u\|_{\tilde{W}^2_{1,0}}.$$
			Thus, Lemma \ref{lm2} holds for $0<\gamma \leq 1$ and we can continue inductively to prove Lemma \ref{lm2} for all $\gamma>0$.
			
			We prove Lemma \ref{lm2} in the case $\gamma<0$ by duality, using the fact that $P$ is self-adjoint.
			\begin{align*}
				\|\langle x \rangle^\gamma P^{-1}(f-\overline{f})\|_{L^2}
				=&\sup_{h\in L^2, \; \|h\|_{L^2}=1} \langle \langle x \rangle^\gamma P^{-1}(f-\overline{f}),h\rangle\\
				=&\sup_{h\in L^2, \; \|h\|_{L^2}=1} \langle \langle x \rangle^\gamma P^{-1}(f-\overline{f}),h-\overline{h}\rangle\\
				=&\sup_{h\in L^2, \; \|h\|_{L^2}=1} \langle f-\overline{f},P^{-1}(  \langle x \rangle^\gamma (h-\overline{h}))\rangle\\
				\lesssim & \sup_{h\in L^2, \; \|h\|_{L^2}=1} \|\langle x \rangle^\gamma(f-\overline{f})\|_{L^2}\| \langle x \rangle^{-\gamma}P^{-1}(  \langle x \rangle^\gamma (h-\overline{h}))\|_{L^2}
			\end{align*}
			We can apply Lemma \ref{lm2} for $-\gamma>0$
			\begin{align*}
				\|\langle x \rangle^\gamma P^{-1}(f-\overline{f})\|_{L^2}
				\lesssim& \|f-\wb f\|_{\tilde{W}^2_{0,\gamma}}\sup_{h\in L^2, \; \|h\|_{L^2}=1} \|\langle x \rangle^\gamma (h-\overline{h})\|_{W^2_{0,-\gamma}}\\
				\lesssim& \|f-\wb f\|_{\tilde{W}^2_{0,\gamma}}
			\end{align*}
			We have proved the desired weighted $L^2$ estimate. It remains to show the full $\tilde W^2_{2,\gamma}$ estimate. For this, we use again \eqref{eqgamma}, which implies that
			$$\|u-\wb u\|_{\tilde{W}^2_{2,\gamma}} \lesssim \|f-\wb f\|_{\tilde{W}^2_{0,\gamma}}+\|u-\wb u \|_{\tilde{W}^2_{1,\gamma-1}}+\|u-\wb u \|_{\tilde{W}^2_{0,\gamma-2}}.$$
			We can use interpolation to estimate $\|u-\wb u \|_{\tilde{W}^2_{1,\gamma-1}}$.  We have
			\begin{align*}
				\|u-\wb u \|_{\tilde{W}^2_{1,\gamma-1}} &\lesssim \|(u-\wb u )\langle x \rangle^{\gamma-1}\|_{H^1}\\
				&\lesssim \ep \|(u-\wb u )\langle x \rangle^{\gamma-1}\|_{H^2}+\frac{1}{\ep}\|(u-\wb u )\langle x \rangle^{\gamma-1}\|_{L^2}\\
				&\ep \|u-\wb u \|_{\tilde{W}^2_{2,\gamma}}+\frac{1}{\ep}\|(u-\wb u )\|_{\tilde{W}^2_{0,\gamma}}
\end{align*}
			We obtain
			$$
			\|u-\wb u\|_{\tilde{W}^2_{2,\gamma}} \lesssim \|f-\wb f\|_{\tilde{W}^2_{0,\gamma}}+\ep \|u-\wb u\|_{\tilde{W}^2_{2,\gamma}} +C(\ep)\|u-\wb u \|_{\tilde{W}^2_{0,\gamma-1}}.$$
			We can use the $L^2$ estimate we have proved and absorb the term $\ep \|u-\wb u\|_{\tilde{W}^2_{2,\gamma}} $ to write
			$$\|u-\wb u\|_{\tilde{W}^2_{2,\gamma}} \lesssim \|f-\wb f\|_{\tilde{W}^2_{0,\gamma}},$$
			which concludes the proof of Lemma \ref{lm2}.
		\end{proof}
		
		\par The second lemma we need for Proposition \ref{lms1} follows. The proof uses harmonic analysis applied to singular integrals.

		\begin{lm}\label{lmh} We recall that $|x,\theta|=(x_1^2+..+x_n^2+(e^{i\theta_1}-1)^2+..+(e^{i\theta_m}-1)^2)^\frac{1}{2}$. We have the following useful bounds on $h$, defined by \eqref{defh}.
			\par $(i)$ For $0<|x,\theta|<1$,
			\begin{equation*}
				| h(x,\theta)|\lesssim \frac{1}{|x,\theta|^{d}}\quad \text{ and }\quad |\nabla h(x,\theta)|\lesssim \frac{1}{|x,\theta|^{d+1}}.
			\end{equation*}
			\par $(ii)$ For $|x,\theta|\geq 1$, we have, for $|k|\leq 1$ and $N\in \m N$
			\begin{equation*}
				|\nabla ^k h(x,\theta)|\leq \frac{C_N}{|x,\theta|^N}.
			\end{equation*}
		\end{lm}
		
		
		\begin{proof}[Proof of Lemma \ref{lmh}] $(i)$ {This proof is an adaptation of \cite{Ste93}, where periodic directions are added.} In order to obtain the desired estimate, we divide the $\xi$ space into suitable spherical shells. Given a smooth cut-off function $\delta_0$ defined on $\R^n$ with $\delta_0(\xi)=1$ for $|\xi|\leq 1$ and $\delta_0(\xi)=0$ on $|\xi|\geq 2$, we define a {difference} function $\delta_1(\xi)=\delta_0(\xi)-\delta_0(2\xi)$ and $\delta_j=\delta_1(2^{-j}\xi)$ for $j\geq 1$. We obtain the following partition of unity
			\begin{equation*}
				1=\delta_0(\xi)+\sum_{j=1}\delta_j(\xi). 
			\end{equation*}
			We define $a_0(\xi,k)=a(\xi,k)\delta_0(\xi)$ and $a_j(\xi,k)=a(\xi,\theta)\delta_j(\xi)$, supported in $2^j\leq |\xi|\leq 2^{j+1}$, for $j\in\m N^*$. Furthermore, let 		
			$$h_j(x,\theta)=\q F^{-1}_{x,\theta}(a_j(\xi,k)),$$
			for all $j\in\m N$. In order to obtain estimates on $h$ and $\nabla h$ away from the origin, we study quantities of the type
			$$(e^{i\theta}-1)^\alpha x^\beta h_j\quad\text{ and }\quad(e^{i\theta}-1)^\alpha x^\beta \nabla_{x,\theta}h_j,$$
			where $\alpha,\beta$ are multi-indices and $(e^{i\theta}-1)^\alpha=(e^{i\theta_1}-1)^{\alpha_1}...(e^{i\theta_{m}}-1)^{\alpha_m}$.
			We note that taking a derivative with respect to $x$ or $\theta$ is equivalent to multiplying by an order one polynomial in $\xi,k$ in the Fourier side: 
			$$ \nabla_{x,\theta}h_j=F^{-1}_x\left(\sum_{\substack{k\in \m Z^{m*}}}p_1(\xi,k)a(\xi,k)\delta_j(\xi)e^{ik\cdot\theta}\right).$$
			Multiplying by $x^\beta$ is equivalent to deriving $\beta$ times with respect to $\xi$ on the Fourier side. Since we are deriving homogeneous functions, $\nabla_\xi^\beta (p_1(\xi,k)a(\xi,k))$ is homogeneous of order $1-|\beta|$.
			For the multiplication by $(e^{i\theta_i}-1)$, we note that we have a telescopic sum. For any function $p(\xi,\theta)$, we have
			\begin{align}
				\label{terme}&	\sum_{\substack{ k\in \m Z^{m*}}}(e^{i\theta_1}-1)p(\xi,k)e^{ik\cdot\theta} \\
				\nonumber& =  \sum_{\substack{k'\in \m Z^{m-1}}}e^{ik'\cdot\theta'}\sum_{k_1 \in \m Z, (k_1,k')\neq (0,0)} (e^{i\theta_1}-1)p(\xi,k)e^{i\theta_1 k_1} \\
				\label{terme1}& =   \sum_{\substack{k'\in \m Z^{m-1}}}e^{ik'\cdot\theta'}\sum_{k_1\in \m Z, \;(k_1,k')\neq (0,0),\; (k_1-1,k')\neq (0,0)} e^{i\theta_1 k_1}(p(\xi,k_1-1,k')-p(\xi,k_1,k'))\\
				\label{terme2}&\qquad - p(\xi,1,0)e^{i\theta_1}+p(\xi,-1,0).
			\end{align}
			Here, we have written $\theta=(\theta_1,\theta')$, $\theta'\in\m T^{m-1}$. We can write
			$$p(\xi,k_1-1,k')-p(\xi,k_1,k')= \int_{k_1-1}^{k_1} \nabla_{u}p(\xi,u,k')\,du,$$
			{and if $p$ is homogeneous of degree $l$, then $\nabla_{u}p(\xi,u,k')$ is homogeneous of degree $l-1$.  Multiplying \eqref{terme} by another factor $(e^{i\theta_i}-1)$, we can apply the same procedure to separate the term \eqref{terme1} into a term which is a sum over $k\in \m Z^m$ of homogeneous functions of order $l-2$, plus a term which is of order $l-1$, but which is not summed over $k$.  However, the term  \eqref{terme2}, which is of degree $l$, does not lose degrees of homogeneity with a multiplication by $(e^{i\theta_i}-1)$, so we keep this factor unchanged.  We can iterate the procedure, and note that for integrable functions we can always estimate a discrete sum by a continuous integral. We write }
			\begin{align*}
				|(e^{i\theta}-1)^\alpha x^\beta \nabla_{x,\theta}h^\gamma_j|\leq &\int_{supp\,\delta_j} \int_{\rho \in \m R^m, |\rho|\geq 1} p_{1-|\alpha |-|\beta|}(\xi,\rho)\,d\rho d\xi\\
				&+ \sum_{|\gamma| <|\alpha|}|(e^{i\theta}-1)^\gamma| \int_{supp\,\delta_j} p_{2-|\alpha|-|\beta|+|\gamma|}(\xi,1,0)\,d\xi,
			\end{align*}
			where we recall that $ p_{1-|\alpha|-|\beta|}$ and $p_{2-|\alpha|-|\beta|+|\gamma|}$ are homogeneous functions. The second term represents the remainder after we extract the telescopic sums. Similarly, we obtain 
			\begin{align*}
				|(e^{i\theta}-1)^\alpha x^\beta h_j|\leq &\int_{supp\,\delta_j} \int_{\rho \in \m R^m, |\rho|\geq 1} p_{-|\alpha |-|\beta|}(\xi,\rho)\,d\rho d\xi\\
				&+ \sum_{|\gamma| <|\alpha|}|(e^{i\theta}-1)^\gamma| \int_{supp\,\delta_j} p_{1-|\alpha|-|\beta|+|\gamma|}(\xi,1,0)\,d\xi.
			\end{align*}
			In what follows, the estimates for $\nabla h_j$ and $h_j$ are obtained in the exact same way. In order to avoid the introduction of unnecessarily heavy notation, we only write the proof for $h_j$. 
			\par By writing $\xi =|\xi|\omega$, we get
			\begin{align*}
				& \int_{\rho \in \m R^m, |\rho|\geq 1} p_{-|\alpha |-|\beta|}(\xi,\rho)d\rho\\
				&\lesssim  \int_{\rho \in \m R^m, |\rho|\geq 1} |\xi|^{-|\alpha|-|\beta|} p_{-|\alpha |-|\beta|}(\omega,\frac{\rho}{|\xi|})\,d\rho \\
				&\lesssim  \int_{\rho \in \m R^m, |\rho|\geq \frac{1}{|\xi|}} |\xi|^{m-|\alpha|-|\beta|} p_{-|\alpha |-|\beta|}(\omega,\rho)\,d\rho \\
				&\lesssim  \int_{\rho \in \m R^m}|\xi|^{m-|\alpha|-|\beta|} (1+|\rho|)^{-|\alpha|-|\beta|}\,d\rho \\
				&\lesssim |\xi|^{m-|\alpha|-|\beta|} 
			\end{align*}
			where we have used $|\alpha|+|\beta|>m$. Integrating on $supp\,\delta_j$, where $2^j\leq |\xi|\leq 2^{j+1}$, we obtain
			$$ \int_{supp\,\delta_j} \int_{\rho \in \m R^m, |\rho|\geq 1} p_{-|\alpha|-|\beta|}(\xi,\rho)\,d\rho d\xi \leq 2^{j(d-|\alpha|-|\beta|)}$$
			and, for the remainder term,
			$$|(e^{i\theta}-1)^\gamma|\int_{supp\,\delta_j} p_{1-|\alpha|-|\beta|+|\gamma|}(\xi,1,0)\,d\xi \leq |(e^{i\theta}-1)^\gamma| 2^{j(1+n-|\alpha|-|\beta|+|\gamma|)}.$$
			Finally, we are able to draw profit from the partitioning of the frequency space we previously introduced. We denote $h_j(x,\theta)=h_j^{princ}+\sum_\gamma h_j^\gamma$ for all $j$, where
			\begin{align*}
				|x^\beta(e^{i\theta}-1)^\alpha h_j^{princ}|&\leq 2^{j(d-|\alpha|-|\beta|)},\\
				|x^\beta (e^{i\theta}-1)^\alpha h_j|&\leq |(e^{i\theta}-1)^\gamma |2^{j(1+n-|\alpha|-|\beta|+|\gamma|)}.
			\end{align*}
			We fix $(x,\theta)$. As we are interested in estimations on $h=\sum_j h_j$, we split the sum in the following way:
			\begin{equation*}
				\sum_j h_j^{princ}(x,\theta)=\sum_{2^j\leq\frac{1}{|x,\theta|}}h_j^{princ}+\sum_{2^j\geq \frac{1}{|x,\theta|}}h_j^{princ}.
			\end{equation*}
			For the first part, fixing $|\alpha|+|\beta|=d-1$ leads to
			\begin{equation*}
				\sum_{2^j\leq\frac{1}{|x,\theta|}}h_j^{princ}\leq \frac{1}{|x,\theta|^{d-1}}\sum_{2^j\leq \frac{1}{|x,\theta|}}2^j\lesssim \frac{1}{|x,\theta|^{d}}.
			\end{equation*}
			For the second part, we fix $|\alpha|+|\beta|=d+1$ and obtain
			\begin{equation*}
				\sum_{2^j\geq \frac{1}{|x,\theta|}}h_j^{princ}\leq \frac{1}{|x,\theta|^{d+1}} \sum_{2^j\geq \frac{1}{|x,\theta|}} 2^{-j}= \frac{1}{|x,\theta|^{d+1}} \sum_{2^{-j}\leq |x,\theta|} 2^{-j}\leq \frac{1}{|x,\theta|^{d}}.
			\end{equation*}
			Similarly, we study the remainder terms $h^{\gamma}$, which can be written as
			\begin{equation*}
				\sum_j h^\gamma_j(x,\theta)=\sum_{2^j\leq\frac{1}{|x,\theta|}}h^\gamma_j+\sum_{2^j\geq\frac{1}{|x,\theta|}} h^\gamma_j.
			\end{equation*}
			We denote $A=|\alpha|+|\beta|$. For $A<1+n$ we have
			\begin{align*}
				\sum_{2^j\leq\frac{1}{|x,\theta|}}h^\gamma_j& \leq\frac{1}{|x,\theta|^A}\sum_{2^j\leq\frac{1}{|x,\theta|}}2^{j(1+n-A+|\gamma|)}|(e^{i\theta}-1)^\gamma|\\
				& \leq\frac{1}{|x,\theta|^A}\sum_{2^j\leq\frac{1}{|x,\theta|}} 2^{j(1+n-A)}\\
				& \leq \frac{1}{|x,\theta|^{1+n}}\\
				& \leq \frac{1}{|x,\theta|^{d}}.
			\end{align*}
			as $|(e^{i\theta}-1)^\gamma| 2^{j|\gamma|}\leq 1.$
			Finally, with $A>1+n+|\gamma|$
			\begin{align*}
					\sum_{2^j\geq\frac{1}{|x,\theta|}}h^\gamma_j&\lesssim
				\frac{1}{|x,\theta|^A}\sum_{2^j\geq \frac{1}{|x,\theta|}}2^{j(1+n+|\gamma|-A)}|(e^{i\theta}-1)^\gamma|\\
				&\lesssim \frac{1}{|x,\theta|^{|1+n+|\gamma||}}|(e^{i\theta}-1)^\gamma|\\
				&\lesssim \frac{1}{|x,\theta|^{1+n}}\left( |\frac{e^{i\theta}-1|}{|x,\theta|} \right)^\gamma\\
				&\lesssim \frac{1}{|x,\theta|^{d}}.
			\end{align*}
			\par $(ii)$ The argument for the decay at infinity is quite straightforward. For $N$ sufficiently large, $\nabla_\xi^N a\in L^1.$ This implies that $|x|^N|h(x)|\leq C_N,$ where $C_N$ is a constant depending on $N$. Similarly, $|x|^N|\nabla h(x)|\leq C_N'$. 
		\end{proof}
		We are ready to prove Proposition \ref{lms1}. A useful reference is the $L^p$ estimate result of Theorem $3$, Section $5$, Chapter $1$ of \cite{Ste93}, together with the Remark $(iii)$ of Section $7.4$. As all other results in Chapter $1$ of the book, the operators can be defined on homogeneous spaces in general, according to Section $1.1$.	
		\begin{proof}[Proof of Proposition \ref{lms1}]
			The idea is to adapt the operators we have seen so far to $L^p$ spaces, and then use existent singular integral theory. We write $\tilde{f}=(f-\bar{f})\langle x\rangle^\gamma$, so that $\tilde{f}\in L^p$ if and only if $f-\bar{f}\in \tilde{W}^p_{0,\gamma}$ and $||\tilde{f}||_{L^p}=||f-\bar{f}||_{\tilde{W}^p_{0,\gamma}}$. Let
			\begin{equation*}
				T(\tilde{f})=\int_{\R^n\times \m T^m}K\left((x,\theta),(y,\chi)\right)\tilde{f}(y,\chi)\,dyd\chi,
			\end{equation*}
			where $K\left((x,\theta),(y,\chi)\right)=\frac{\langle x\rangle^\gamma}{\langle y\rangle^\gamma}h\left((x,\theta)-(y,\chi)\right)$. So, by \eqref{h def}, $T(\tilde{f})=\nabla^2_{x,\theta}(u-\bar{u})\langle x\rangle^\gamma$. In this proof, as mentioned before, we concentrate on obtaining the estimates for the second derivatives of $u-\bar{u}$. The lower order derivatives result from the same procedure applied to the the corresponding operator $T$.
			\par According to \cite{Ste93}, we need to check the following properties. Given $\tilde{f}\in L^2\cap L^p$, we verify that
			\begin{enumerate}
				\item $\displaystyle ||T(\tilde{f})||_{L^2}\leq A||\tilde{f}||_{L^2}, $
				\item $\displaystyle K((x,\theta),(y,\chi))\leq \frac{A}{|(x,\theta)-(y,\chi)|^d}, $
				\item Let $c>1$ Then $$ \int_{|(x,\theta)-(y,\chi)|\geq c\rho}|K((x,\theta),(y,\chi))-K((x,\theta),(\bar{y},\bar{\chi}))|\,dxd\theta\leq C,$$
				where $|(y,\chi)-(\bar{y},\bar{\chi})|\leq\rho$ for all $(y,\chi)\in\R^n\times\m T^m$, $\rho>0$. The same must be true with the roles of $(x,\theta)$ and $(y,\chi)$ reversed.
			\end{enumerate}
			\par If this holds, then $||T(\tilde{f})||_{L^p}\lesssim ||\tilde{f}||_{L^p}$, and the result extends to all $\tilde{f}\in L^p$. 
			\medskip
			\par The first condition is immediately satisfied thanks to Lemma \ref{lm2}. 
			\par For the second condition, if $|(x,\theta)-(y,\chi)|\leq 1$, then $\frac{\langle x\rangle^\gamma}{\langle y\rangle{^\gamma}}\leq 2^{\gamma}$ and Lemma \ref{lmh} $(i)$ imply the desired result. For $|(x,\theta)-(y,\chi)|>1$, we apply Lemma \ref{lmh} $(ii)$ with $N>d+|\gamma|$. 
			\par For the last condition, we note that
			\begin{align*}
				\nabla_x K\left((x,\theta),(y,\chi)\right)&=\gamma\frac{x\langle x\rangle^{\gamma-2}}{\langle y \rangle^\gamma}h\left((x,\theta)-(y,\chi)\right)+\frac{\langle x\rangle^{\gamma}}{\langle y \rangle^\gamma}\nabla_x h\left((x,\theta)-(y,\chi)\right),\\
				\nabla_\theta K\left((x,\theta),(y,\chi)\right)&=\frac{\langle x\rangle^{\gamma}}{\langle y \rangle^\gamma}\nabla_\theta h((x,\theta)-(y,\chi)),\\
				\nabla_y K\left((x,\theta),(y,\chi)\right)&=\gamma\frac{y\langle x\rangle^{\gamma}}{\langle y \rangle^{\gamma+2}}h\left((x,\theta)-(y,\chi)\right)+\frac{\langle x\rangle^{\gamma}}{\langle y \rangle^\gamma}\nabla_y h\left((x,\theta)-(y,\chi)\right),\\
				\nabla_\chi K\left((x,\theta),(y,\chi)\right)&=\frac{\langle x\rangle^{\gamma}}{\langle y \rangle^\gamma}\nabla_\xi h((x,\theta)-(y,\chi)).\\
			\end{align*}
			We define $\Omega_1=\{\left ((x,\theta),(y,\chi) \right)\in (\R^n\times\m T^m)^2\, |\, c\rho\leq|(x,\theta)-(y,\chi)|\leq 1 \}$ and {$\Omega_2=\{ \left( (x,\theta),(y,\chi) \right)\in (\R^n\times\m T^m)^2\, | \,1, c\rho \leq |(x,\theta)-(y,\chi)|\}$}. Depending on the value of $\rho$, $\Omega_1$ might be a null set. 
			\par First, we note that for $((x,\theta),(y,\chi)) \in \Omega_1$,
			\begin{align*}
				&|K\left((x,\theta),(y,\chi)\right)-K\left((x,\theta),(\bar{y},\bar{\chi})\right)|\\
				&\quad\quad\lesssim \int_0^1\left|((y-\bar y)\cdot\nabla_y K + (\chi-\bar \chi)\cdot\nabla_\chi K)\left((x,\theta),\left(y+t(\bar y - y),\chi+t(\bar \chi - \chi)\right)\right)\right|dt\\
				&\quad\quad\lesssim \frac{|(y,\chi)-(\bar y,\bar \chi)|}{|(x,\theta)-(y,\chi)|^{d+1}}.
			\end{align*}
			Consequently,
			\begin{equation*}
				\int_{\Omega_1} |K\left((x,\theta),(y,\chi)\right)-K\left((x,\theta),(\bar{y},\bar{\chi})\right)| dxd\theta \lesssim \rho\int_{c\rho \leq |(x',\theta')|\leq 1} \frac{1}{|(x',\theta')|^{d+1}}dx'd\theta' \lesssim 1.
			\end{equation*}
			For $\Omega_2$, we use the strong decay properties of $h$ and $\nabla h$ to obtain the desired results. By symmetry, the same must be true with the roles of $(x,\theta)$ and $(y,\chi)$ reversed.
		\end{proof}
		
		\subsection{Proof of Theorem \ref{thlin}  for $P$}\label{secP}
		{We recall that $P$ is a second order homogeneous elliptic operator with constant coefficients.}  We prove the following properties for $P$.
		\begin{itemize}
			\item $P: W^p_{2,\delta,\gamma} \to W^P_{0,\delta+2,\gamma}$ is a continuous map.
			\item $P$ has a finite dimensional kernel.
			\item The a-priori estimate \eqref{estimp} holds.
			\item $P$ has closed range.
		\end{itemize}
		Let $u \in C^\infty\cap W^p_{2,\delta,\gamma}$.  Since $P$ has constant coefficients, we have $\wb{Pu} =P \wb u$, and 
		\begin{align*}
			\|P u \|_{W^p_{0,\delta+2,\gamma} }=&\int_{\m R^n} | \overline{Pu}|^p \langle x\rangle^{p(\delta+2)}\,dx
			+ \int_{\m R^n\times \m T^m} | (P u-\wb{P u})|^p \langle x\rangle^{p\gamma}dxd\theta,\\
			\leq &\sum_{0\leq |\beta|\leq 2} \left(\int_{\m R^n} |\partial^{\beta} u|^p \langle x\rangle^{p(\delta+|\beta|)}\,dx
			+ \int_{\m R^n\times \m T^m} |\partial^\beta (u- \wb u)|^p\langle x\rangle^{p\gamma}\,dxd\theta \right) \\
			\lesssim &\|u\|_{W^p_{2,\delta,\gamma}}.\\
		\end{align*}
		Consequently, the operator $P$ can be extended by density to an operator $ W^p_{2,\delta,\gamma}\to W^p_{0,\delta+2,\gamma}$.
		\medskip
		\par We study the kernel of $P$. Let  $\delta$ be such that $-\delta-\frac{n}{p} \notin \m N$  and $u\in W^p_{2,\delta,\gamma}$ such that $P u= 0$. Then $u \in C^\infty$  and, by taking the average over $\m T^m$, we see that  $P_0 \wb{u} = 0$, where $P_0$ is the operator  obtained from $P$ by keeping only the derivatives with respect to $x$. Its Fourier symbol is $P(\xi,0)$, which is invertible for $\xi \neq 0$, so $P_0$ is elliptic.  Consequently, $\bar{u}$ belongs to $Ker P_0$, which by Theorem \ref{thcb} is finite dimensional.
		Moreover,  $P (u-\wb{u})=0$, so Proposition \ref{lms1} implies $u-\wb{u}=0$. This shows that $Ker P$ is finite dimensional.
		\medskip	
		\par We prove estimate \eqref{estimp}. We have,
		\begin{align*}
			\|u\|_{W^p_{2,\delta,\gamma}}&\lesssim \|\bar{u}\|_{W^p_{2,\delta}}+\|u-\wb u \|_{W^p_{2,\delta,\gamma}}\\
			&\lesssim  \|P_0\bar{u}\|_{W^p_{0,\delta+2}} +\|\bar{u}\|_{L^p(B_R)}+\|P(u-\wb u)\|_{\tilde W^p_{2,\gamma}}\\
			&\lesssim \|\bar{u}\|_{L^p(B_R)}+ \|P u \|_{W^p_{0,\delta+2,\gamma}},
		\end{align*}
		where $R$ is chosen from Theorem \ref{thcb}.		
		
		\par From the fact that $Ker P$ is finite dimensional, and from the estimate above, we can show that $P$ has closed range. We can follow the proof of Theorem 1.10 in \cite{Bar86}). As $Ker(P)$ is finite dimensional, we can write $W^p_{2,\delta,\gamma}= Ker(P)+Z$, where $Z$ is a closed space with $Ker(P)\cap Z=\emptyset$. Moreover, from the estimate above, there exists $C$ such that for all $z\in Z$,
		\begin{equation}\label{estinter}	\|z\|_{W^p_{2,\delta,\gamma}} \leq C \|P z \|_{W^p_{0,\delta+2,\gamma}}.
		\end{equation}
		Let us consider a sequence $u_n \in W^p_{2,\delta,\gamma}$ such that $P(u_n)$ converges in $W^p_{0,\delta+2,\gamma}$. We write $u_n =w_n+z_n$ with $w_n \in Ker(P)$ and $z_n \in Z$.  From \eqref{estinter} we obtain that $z_n$ is a Cauchy sequence, so it converges to some $z$.  Consequently, $P(u_n)$ converges to $P(z)$ and the range of $P$ is closed.

		\subsection{Proof of Theorem \ref{thlin} for $L$}\label{secthlin}
		In this section, we assume that $L$ is asymptotic to $P$ in $W^p_{2,\sigma,\lambda}$.
		We extend the elliptic theory we have established in the previous section to prove Theorem \ref{thlin},  following the proof of Theorem 1.10 in Bartnik \cite{Bar86}.
		\par In order to check that $L:W^p_{2,\delta,\gamma}\to W^p_{0,\delta+2,\gamma}$, we need the multiplication properties
		$$W^p_{2,\sigma,\lambda} \times W^p_{0,\delta+2,\gamma}\subset W^p_{0,\delta+2,\gamma},$$
		$$W^p_{1,\sigma+1,\lambda} \times W^p_{1,\delta+1,\gamma} \subset W^p_{0,\delta+2,\gamma},$$
		$$W^p_{0,\sigma+2,\lambda} \times W^p_{2,\delta,\gamma} \subset W^p_{0,\delta+2,\gamma}.$$
		Thanks to Proposition \ref{multemb prop}, we can verify them, if the following conditions are satisfied
		\begin{equation}\label{conditions}
			\frac{n+m}{p}<2, \quad -\frac{n}{p}<\sigma, \quad \delta-\lambda+2<\gamma<\delta+\lambda+\frac{n}{p}.
		\end{equation}
		From now on, we assume \eqref{conditions}. The strategy involves looking at estimates on scalar functions supported in a ball or at infinity.
		We start with the following lemma.
		\begin{lm}\label{lmdec} 	
			If $u$ is supported in $\{r\geq R\}$ then 
			$$\|(L-P) u\|_{W^p_{0,\delta+2,\gamma}} \leq \ep(R)\|u\|_{W^p_{2,\delta,\gamma}},$$ with $\ep(R)\to 0$ as $R\to \infty$.
		\end{lm}
		\begin{proof}
			
			We can find $\tau>0$ such that
			$$\|(L-P) u\|_{W^p_{0,\delta+2,\gamma}}\lesssim (\sum_{|\alpha|\leq 2}\|a^{(\alpha)}\|_{W^p_{|\alpha|,\sigma+2-|\alpha|-\tau,\lambda-\tau}})\|u\|_{W^p_{2,\delta,\gamma}},$$
			so if $u$ is supported in $\{r\geq R\}$ 
			\begin{align*}\|(L-P) u\|_{W^p_{0,\delta+2,\gamma}}&\lesssim (\sum_{|\alpha|\leq 2}\|a^{(\alpha)}\|_{W^p_{|\alpha|,\sigma+2-|\alpha|-\tau,\lambda-\tau}(\{r\geq R\})})\|u\|_{W^p_{2,\delta,\gamma}}\\
				&\lesssim R^{-\tau}(\sum_{|\alpha|\leq 2}\|a^{(\alpha)}\|_{W^p_{|\alpha|,\sigma+2-|\alpha|,\lambda}})\|u\|_{W^p_{2,\delta,\gamma}}.
			\end{align*}
		\end{proof}
		
		We are ready to finish the Proof of Theorem \ref{thlin}.
		\begin{proof}[Proof of Theorem \ref{thlin}]. We write $u=u_0+u_\infty$ where $u_0=\chi(\frac{r}{R})u$, and $u_\infty= \left(1-\chi(\frac{r}{R})\right)u,$ with $\chi$ supported in $B(0,2)$ with $\chi = 1$ on $B(0,1)$. We have
			\begin{align*}
				\|u_\infty\|_{W^p_{2,\delta,\gamma}}\lesssim  &\|\wb u_\infty\|_{W^p_{2,\delta}}+\|u_\infty-\wb u_\infty\|_{W^p_{2,\delta,\gamma}}\\
				\lesssim &\|P_0 \wb u_\infty\|_{W^p_{0,\delta+2}}+\|P (u_\infty-\wb u_\infty)\|_{ W^p_{0,\delta,\gamma}}\\
				\lesssim &\|P u_\infty\|_{W^p_{0,\delta+2,\gamma}}.
			\end{align*}
			{where we used Theorem 1.7 in Bartnik \cite{Bar86} to write \footnote{{This result uses the fact that $\Delta$ is an isomorphism on homogeneous weighted Sobolev spaces, where the weight $(1+r^2)^\frac{1}{2}$ is replaced by $r$. As in \cite{Bar86}, the result can easily be generalized to elliptic operators with constant coefficients.}}
				$\|\wb u_\infty\|_{W^p_{2,\delta}(\m R^n)}\leq \|P_0 \wb u_\infty\|_{W^p_{0,\delta+2}(\R^n)}$  and Lemma \ref{lms1} for the non zero modes.}
			Consequently,
			$$\|u_\infty\|_{W^p_{2,\delta,\gamma}} \lesssim \|L u_\infty\|_{W^p_{0,\delta+2,\gamma}}+ C\ep(R)\|u_\infty\|_{W^p_{2,\delta,\gamma}}$$
			with $\ep(R)\to 0$ as $R \to \infty$. We choose $R$ big enough to write
			$$\|u_\infty\|_{W^p_{2,\delta,\gamma}} \lesssim \|L u_\infty\|_{W^p_{0,\delta+2,\gamma}}.$$
		We can write
			\begin{align*}
				\|L u_\infty\|_{W^p_{0,\delta+2,\gamma}}
				\lesssim &\|(1-\chi(\frac{r}{R}))L u + [L,(1-\chi(\frac{r}{R}))]u\|_{W^p_{0,\delta+2,\gamma}}\\
				\lesssim & \|L u\|_{W^p_{0,\delta+2,\gamma}}+C(R)\|u\|_{W^{1,p}(A_R)},
			\end{align*}
			where $A_R$ is the annulus $R\leq r \leq 2R$. Moreover, the interior estimate given by Theorem 8.8 in Gilbarg Trudinger \cite{giltru84} gives
			$$\|u_0\|_{W^{2,p}(B_{2R})} \lesssim \|L u_0 \|_{L^p(B_{3R})}+\|u_0\|_{W^{1,p}(B_{3R})} \lesssim \|L u\|_{L^p(B_{3R})}+\|u\|_{W^{1,p}(B_{3R})}.$$
			By interpolation, on the usual Sobolev spaces $W^{k,p}(B_{3R})$ we obtain
			$$\|u\|_{W^p_{2,\delta,\gamma}} \lesssim C(R)\|L u\|_{W^p_{0,\delta+2,\gamma}}+\|u\|_{L^p(B_{3R})}.$$
			This proves \eqref{estimp}	and yields that the unit ball of $ker L$ is compact in $W^{p}_{2,p,\delta}$, which implies that $ker L$ is finite dimensional. 
			The argument to show that the image is closed is the same as for $P$.
		\end{proof}
		

		\subsection{The injectivity of the Laplacian and conformal Laplacian}\label{injsec}
		\subsubsection{A maximum principle for $\Delta_g$}
		In the case of solutions which decay at infinity, the injectivity of $\Delta_g$ can be obtained from the maximum principle. The following result is a direct consequence of the maximum principle stated in Theorem 8.19 of \cite{giltru84} in the compact case.
		\begin{lm}\label{positivity cor} Let $g$ be a Riemannian metric on $M$ with bouded coefficients and $\phi \in W^{1,2}_{loc}$ be such that
		\begin{equation*}
			\Delta_g \varphi + h\varphi \geq 0,
		\end{equation*}
		where $h$ is a non negative bounded scalar field. Suppose moreover that there exists $A\in\R$ such that $u\to A$ at infinity. 
		\begin{enumerate}
			\item If $A>0$, then there exists $0<\varepsilon\leq A$ such that $\varphi\geq \varepsilon$ on the manifold.
			\item If $A=0$, then $\varphi\geq 0$.
		\end{enumerate}
	\end{lm}
	\begin{proof}
		Let $0<\varepsilon_0<A$. In a neighborhood of infinity, we have that $\phi \geq A-\varepsilon0$, consequently, if the property we want to prove were not true, then there would exist a compact $K$ such that $\inf_K \phi = \inf_{M} \phi \leq 0$: by the maximum principle (Theorem 8.19 in \cite{giltru84}), this would imply that $\phi$ is constant, equal to $A$, which is a contradiction.
		\par Suppose that $\phi$ takes a negative value $\lambda<0$. Let $0<\varepsilon_0<|\lambda|$. In a neighborhood of infinity we have $-\ep_0 \leq \phi$. Consequently there would exist a compact $K$ such that $\inf_K \phi = \inf_{M} \phi <0$ which again implies that $\phi$ is a constant, equal to $0$, which is again a contradiction.
	\end{proof}
		From this Lemma, we obtain the following corollary, which is our desired injectivity result for $\Delta_g$.
		\begin{cor}\label{corinj}
			Let $g-\zeta \in W^p_{2,\sigma,\lambda}$. Let us assume that $(\sigma,\lambda,\delta,\gamma)$ satisfy the conditions \eqref{conditions}, together with the hypothesis
			\begin{equation}\label{dec}0<\gamma, \quad -\frac{n}{p}<\delta.
			\end{equation}
			Let $h\in W^p_{0,\delta+2,\gamma}$ be a non negative bounded function.
			Then $\Delta_g +h : W^p_{2,\delta,\gamma} \to W^p_{0,\delta+2,\gamma}$ is injective.
		\end{cor}
		\begin{proof}
			Let $u\in W^p_{2,\delta,\gamma}$ be such that $\Delta_g u +h u=0$. The hypothesis \eqref{dec} implies that $u \to 0$ at infinity. Therefore, Lemma \eqref{positivity cor} implies that $u\geq 0$. By considering $-u$, we obtain $u=0$, so $\Delta_g +h$ is injective.
		\end{proof}

		\subsubsection{Conformal Killing Vector Fields That Vanish at Infinity.} 		
		We recall that we are interested in metrics $g_{ij}-\zeta_{ij}\in W^p_{2,\sigma,\lambda}$. We study the injectivity of the conformal Laplacian
		$$\overrightarrow{\Delta}_g : W^p_{2,\delta,\gamma} \to W^p_{0,\delta+2,\gamma}.$$
		We always assume that $\delta,\gamma,\sigma,\lambda$ satisfy the set of hypothesis \eqref{conditions} that we recall
		$$		  \frac{n+m}{p}<2, \quad -\frac{n}{p}<\sigma, \quad \delta-\lambda+2<\gamma<\delta+\lambda+\frac{n}{p}.$$
		
		\par In order to satisfyingly solve the momentum constraint, we first want to check that $X\equiv 0$ is the only conformal Killing vector field in $W^p_{2,\delta,\gamma}$. This is similar to the classical case of the asymptotically Euclidean manifold. 
		We start with a general property of conformal Killing fields, which can be found in \cite{christodoulou_omurchadha}. We give the proof for the sake of completness.
		\begin{lm}\label{third deriv}
			Let $X$ be a conformal Killing vector field in any manifold $M$ of dimension $n>2$. Then, the following results hold true in local coordinates,
			\begin{equation*}
				\nabla_\lambda\nabla_\gamma\nabla_\alpha X_\beta=R X+\nabla(RX),
			\end{equation*}
			where $RX$ is an operator comprised of linear terms of the Riemann curvature tensor, and $d=n+m$ is the dimension of the manifold.
		\end{lm}
		\begin{proof}
			We recall that a conformal Killing vector field $X$ verifies
			\begin{equation}\label{conformal killing}
				\nabla_\alpha X_\beta +\nabla_\beta X_\alpha=\frac{2}{d}div_g X g_{\alpha\beta}
			\end{equation}
			and that, by the definition of the Riemannian curvature tensor, we commute two derivatives of $X$ to get
			\begin{equation}\label{riem}
				\nabla_\mu \nabla_\nu X=\nabla_\nu\nabla_\mu X + RX.
			\end{equation}
			We use the notation $RX$ to designate (any) linear operator in the Riemannian tensor. We do not specify the exact form of $RX$ to lighten the proof. The operator $RX$ is not always the same in every line.
			If we derive \eqref{conformal killing} by $\nabla^\alpha$, we obtain
			\begin{equation*}
				\nabla^\alpha\nabla_\alpha X_\beta +\nabla^\alpha\nabla_\beta X_\alpha-\frac{2}{d}\nabla^\alpha div_g X g_{\alpha\beta}=0,
			\end{equation*}
			and by applying \eqref{riem}, it follows that
			\begin{equation}\label{d div X}
				\left(\frac{2}{d}-1\right)\nabla_\beta div X=\nabla^\alpha\nabla_\alpha X_\beta +RX.
			\end{equation}
			If we derive again by $\nabla_\lambda$, we see that
			\begin{align*}
				\left(\frac{2}{d}-1\right)\nabla_\lambda\nabla_\beta div X & = \nabla_\lambda\nabla^\alpha\nabla_\alpha X_\beta+\nabla(RX) \\
				& =   \nabla^\alpha\nabla_\alpha(-\nabla_\beta X_\lambda+\frac{2}{d}g_{\lambda\beta}div X) +\nabla(RX)+R\nabla X \\
				& = -\nabla_\beta \left(\left(\frac{2}{d}-1\right) \nabla_\lambda div X\right)+\frac{2}{d}g_{\lambda \beta}\Delta div X+\nabla(RX)+R\nabla X \\
				& = -\left(\left(\frac{2}{d}-1\right) \nabla_\lambda \nabla_\beta div X\right)+\frac{2}{d}g_{\lambda \beta}\Delta div X+\nabla(RX)+R\nabla X \\
			\end{align*}
			Contracting by $g^{\lambda \beta}$, we obtain
			$$2\left(\frac{2}{d}-1\right)\Delta dix X = 2\Delta dix X +\nabla(RX)+R\nabla X$$
			which yields
			\begin{equation}\label{eqdiv}4\left(\frac{1}{d}-1\right)\Delta dix X=\nabla(RX)+R\nabla X.
			\end{equation}
			If we derive \eqref{conformal killing} by $\nabla_\gamma$ and again apply \eqref{riem} twice, we obtain
			\begin{align}\label{ddk}
				\nabla_\gamma\nabla_\alpha X_\beta & =  - \nabla_\beta\nabla_\gamma X_\alpha +\frac{2}{d}g_{\alpha\beta}\nabla_\gamma div X + RX\\
				& =  - \nabla_\beta\left( {-}\nabla_\alp X_\gamma {+}\frac{2}{d}g_{\alpha\gamma}div X \right) +\frac{2}{d}g_{\alpha\beta}\nabla_\gamma div X + RX\\
				& =  -\nabla_\alpha\nabla_\gamma X_\beta +\frac{2}{d}g_{\gamma\beta}\nabla_\alpha div X-\frac{2}{d}g_{\alpha\gamma} \nabla_\beta div X+\frac{2}{d}g_{\alpha\beta}\nabla_\gamma div X + RK.
			\end{align}
			so that altogether we obtain
			\begin{equation*}
				\nabla_\lambda\nabla_\gamma\nabla_\alpha X_\beta=\frac{\Delta_g  div X}{d^2\left( \frac{2}{d}-1\right)}(g_{\alpha\beta}g_{\lambda\gamma}-g_{\alpha\gamma}g_{\lambda\beta}+g_{\gamma\beta}g_{\lambda\alpha})+RX+\nabla(RX).
			\end{equation*}
			Applying \eqref{eqdiv} to the above equation gives us the result.
		\end{proof}
		\begin{rk} Note that, in the case of a flat metric, the aforementioned terms $RX$ and $\nabla(RX)$ vanish.
		\end{rk}
		\par First, we look at the manifold $(\R^n\times \m T^m,\zeta)$. 
		\begin{lm}
			Let $X\in W^p_{0,\delta,\gamma}$ be a nontrivial conformal Killing vector field on $(\R^n\times \m T^m,\zeta)$ with $\delta>-\frac{n}{p}$ and $\gamma>0$. Then $X\equiv 0$.
		\end{lm}
		\begin{proof}
			Given that the third derivatives $\nabla^3_\zeta$ is null, the coefficients of the vector field $X$ must all be polynomials in $x$ and $\theta$ of degree at most $2$. However, as $X$ also decays to zero at infinity, this implies that all the coefficients are null.  
		\end{proof}
		We are ready to move on to studying conformal Killing fields on the asymptotically flat manifold $(M,g)$. We denote $E_1=(\R^n\times \m T^m)-(\overline{B_1\times \m T^m})$. {In this part, we follow closely \cite{Max04}}. In order to show that $X$ vanishes in $\R^n\times \m T^m$, we begin by analyzing its behaviour on a neighbourhood close to infinity, and then in an interior domain.
		\begin{lm}\label{lminjext}
			{	Let $g$ be such that $g_{ij}-\zeta_{ij}\in W^p_{2,\sigma,\lambda}$.
				Suppose $X$ is a conformal Killing vector field in $W^p_{0,\delta,\gamma}$.
				In addition to the set of hypothesis \eqref{conditions} we assume
				\begin{equation}\label{condinj}
					\lambda +\frac{n+m}{p}>2, \quad \delta>-\frac{n}{p}, \quad \gamma>0.
				\end{equation}
				Then $X$ vanishes in a neighbourhood of infinity.}
		\end{lm}
		\begin{proof}{The proof requires a rescalling argument. For  $(x,\theta)\in E^1$}, we define $g_m(x,\theta)=g(2^mx,2^m\theta)$. Remember that	
			\begin{align*}
				\bar{g}_m(x) & =  \frac{1}{(2\pi)^k}\int_{\m T^k}g(2^mx,2^m\theta)\\
				& =  {\frac{1}{(2\pi)^k}}\sum_{l_1=0}^{2^m-1}\sum_{l_2=0}^{2^m-1}\dots \sum_{l_k=0}^{2^m-1}\int_{\prod^k_{i=1}[\frac{2\pi l_i}{2^m},\frac{2\pi(l_i+1)}{2^m}]}g(2^m, x 2^m\theta)\\
				& =  \frac{2^{mk}}{(2\pi)^k 2^{mk}}\int_{\m T^k}g(2^mx,\theta)\\
				& =  \bar{g}(2^m x).
			\end{align*}	
			We can say that $||g_m-\zeta||_{W^p_{2,\delta,\gamma}}\to 0$ as $m\to\infty$ as for $l\in\overline{0,2}$, 
			\begin{equation}\label{gm1}
				\int_{\R^n}|\nabla^lg_m|^p\langle x\rangle^{p(\sigma+l)}\,dx\leq (2^m)^{-p\sigma-n}||g||_{W^p_{2,\sigma,\lambda}}
			\end{equation}
			and
			\begin{equation}\label{gm2}
				\int_{\R^n\times \m T^m}|\nabla^l(g_m-\bar{g}_m)|^p\langle x\rangle^{p\lambda}dxd\theta\leq (2^{m})^{-p\lambda-(n+m)+2p}||g||_{W^p_{2,\sigma,\lambda}}.
			\end{equation}
			The hypothesis $\sigma>-\frac{n}{p}$ present in the set of hypothesis \eqref{conditions} ensures that the quantities present in \eqref{gm1} tends to zero, and the additionnal hypothesis $\lambda +\frac{n+m}{p}>2$ ensures that the quantites present in \eqref{gm2} tends to zero, as $m$ tend to infinity.
			\par {We assume that there is no neighborhood of infinity in which $X$ is identically $0$.  Let $\hat{X}_m(x,\theta)=X(2^mx,2^m\theta)$.  With our hypothesis, we have  $||  \hat{X}_m ||_{W^p_{2,\delta,\gamma}}\neq 0$ and we can define the normalized vector field $X_m=\frac{\hat{X}}{|| \hat{X}_m ||_{W^p_{2,\delta,\gamma}}}$. Note that $\mathcal{L}_{g_m}X_m=0$  and $\overrightarrow{\Delta}_{g_{m}}(X_{m})=0$. We use estimate \eqref{estimp}, together with Proposition 5.2 in \cite{Max04} to treat the inner boundary of $E_1$ : 
				\begin{align*}
					&\|X_{m_1}-X_{m_2}\|_{W^p_{2,\delta,\gamma}}\\
					\lesssim&\|\overrightarrow{\Delta}_{g_{m_1}}(X_{m_1}-X_{m_2})\|_{W^p_{0,\delta+2,\gamma}}+\|X_{m_1}-X_{m_2}\|_{L^p(B_R \cap E_1)}+\|\q L_{g_{m_1}}(X_{m_1}-X_{m_2})\|_{W^{1-\frac{1}{p},p}(\partial E_1)}\\
					\lesssim&	\|g_{m_1}-g_{m_2}\|_{W^p_{2,\sigma,\lambda}} +\|X_{m_1}-X_{m_2}\|_{L^p(B_R \cap E_1)}
				\end{align*}
				From the $W^p_{2,\delta,\gamma}$ boundedness of the sequence $(X_m)$, it follows that, up to a subsequence, the vectors $X_m$ converge in $L^p(B_R)$ to some $X_0$.
				which means that $(X_m)$ is a Cauchy sequence in $W^p_{2,\delta,\gamma} $ and $X_m\to X_0$ in $W^p_{2,\delta,\gamma}$. This implies that $X_0$ is a conformal Killing vector field with respect to the $\zeta$ metric, and thus $X_0=0$. Hence, $X_m\to 0$ in $W^p_{2,\delta,\gamma}$, which contradicts $||X_m||_{W^p_{2,\delta,\gamma}}=1$.}
		\end{proof}
		\begin{thm}\label{ckf vanish}Let $g$ be such that $g_{ij}-\zeta_{ij}\in W^p_{2,\delta,\gamma}$.  We assume \eqref{conditions} and \eqref{condinj}.
			Then there exists no nontrivial conformal Killing vector field on $(M,g)$ in $W^p_{2,\delta,\gamma}$.
		\end{thm}
		\begin{proof} {Loosely speaking, the proof of  Theorem 6.4 in \cite{Max04} applies to Theorem \ref{ckf vanish}  since when we zoom on a point of $M$, there is no difference between $\m R^{n+m}$ and $\m R^n \times \m T^n$. But to be precise, we rewrite the proof here}. {So far,  we have proved that $X$ vanishes in a neighborhood of infinity. Consequently we can consider the minimum of the radius $R$ such that $X$ vanishes in $(\m R^n -B_R)\times \m T^n$. If this minimum is zero, then $X$ vanishes everywhere. Otherwise, let $R_0>0$ be that minimum: there exists $x_0$ on $S_{R_0} \times \m T^m$, and $x_k\to x_0$ such that $X(x_k)\neq 0$.  We may choose a set of coordinates centered around $x_0$ such that $g_{\mu\nu}(0)=\delta_{\mu\nu}(0)$ ($x_0$ is the point of coordinate $0$).
				Let $r_k =2|x_k|$: we have $r_k>0$ and $r_k\to 0$ by definition. We construct a sequence of metrics $g_k(x,\theta)=g(\frac{x}{|x_k|}, {\frac{\theta}{r_k}})$ {on the unit ball $B_1$ centered in 0}. Obviously, $g_k\to{\delta}$ in $W^{2,p}(B_1)$. 
			}
			
			{We define a corresponding sequence of vectors $$X_k(x,\theta)=\frac{X(x/r_k,{\theta/r_k})}{\|X(x/r_k,{\theta/r_k})\|_{W^{2,p}}}$$ in $B_1$. Note that the renormalization is possible as the vectors are nontrivial by our supposition. 
				We have
				$$\q L_{g_k}X_k=0, \quad \overrightarrow{\Delta}_{g_k} X_k =0.$$
				Consequently, as in the proof of Lemma \ref{lminjext}, we have
				$$\|X_{k_1}-X_{k_2}\|_{W^{2,p}(B_1)}\lesssim \|g_{k_1}-g_{k_2}\|_{W^{2,p}(B_1)}+\|X_{k_1}-X_{k_2}\|_{L^p(B_1)}.$$
				From the $W^{2,p}$ boundedness of $X_k$, it follows that up to a subsequence, $X_k$ converges strongly in $C^0$ to some $X_0\in W^{2,p}$. From the inequality above, this implies in fact that $X_k$ converges strongly in $W^{2,p}$.
			}	
			
			{	It follows that $X_0$ is a conformal Killing vector field for $\delta$. From the $\mathcal{C}^0(B_1)$ convergence of $X_k$ to $X_0$, and the fact that by definition of $x_0$, the vector field $X_k$ vanish in a fixed open set, then $X_0$ also vanishes in this open set.  Since the coefficients of $X_0$ are polynomials, this implies that $X_0=0$, which contradicts the fact that $||X_k||_{W^{2,p}}=1$.
				This contradicts the fact that $R_0>0$ and implies that $X=0$ on $M$.}
		\end{proof}
		Finally, we can apply this result to study the injectivity of 	$\overrightarrow{\Delta}_g$.
		\begin{prp}\label{prpinj}
			Assume that $\delta,\gamma, \sigma, \lambda$ satisfy the set of hypothesis \eqref{conditions} and \eqref{condinj}. Then $\overrightarrow{\Delta}_g : W^p_{2,\delta,\gamma} \to W^p_{0,\delta+2,\gamma}$ is injective.
		\end{prp}
		\begin{proof}
			The Laplacian operator $\overrightarrow{\Delta}_g$ is injective. Indeed, any smooth compactly supported $X$ in the kernel of $\overrightarrow{\Delta}_g$ is a conformal Killing vector field, as 
			\begin{equation*}
				\int_M X\overrightarrow{\Delta}_g X d\mu_g=\int_M (\mathcal{L}_gX)^2d\mu_g=0.
			\end{equation*}
			By a density argument, the same holds true for $X\in W^p_{2,\delta,\gamma}$. As conformal Killing fields that decay at infinity vanish everywhere (Theorem \ref{ckf vanish}), we conclude that the kernel of $\overrightarrow{\Delta}_g$ is trivial.
		\end{proof}
		\subsection{Proof of Theorem \ref{thinv}}\label{sec iso L}
		
		We recall that $g_{ij}-\zeta_{ij}\in W^p_{2,\sigma,\lambda}$.
		The aim of this section is to prove that the operators $\Delta_g, \overrightarrow{\Delta}_g : W^p_{2,\delta,\gamma} \to W^p_{0,\delta+2,\gamma}$ are isomorphisms, under the conditions \eqref{conditions}, \eqref{condinj}, and the additional condition
		\begin{equation}
			\label{condiso}
			\delta <-\frac{n}{p}+n-2.
		\end{equation}
		To prove Theorem \ref{thinv}, we first consider the case where $g$ is the flat metric, and go to the asymptotically flat case with a continuity argument.
		
		Let  $P$ equal to either $\Delta$ or $\overrightarrow{\Delta}$. Both cases can easily be treated at the same time. From the previous section, we see that, when $\delta>-\frac{n}{p}$ and $\gamma>0$, $P: W^p_{2,\delta,\gamma} \to W^p_{0,\delta+2,\gamma}$ is injective. We also know that $$\wb{Im( P: W^p_{2,\delta,\gamma} \to W^p_{0,\delta+2,\gamma})}=   
		(Ker (P^* : W^{p'}_{0,-\delta-2,-\gamma} \to W^{p'}_{-2,-\delta,-\gamma}))^\perp,$$
		where $p'$ is the dual of $p$. Recall that $W^{p'}_{-2,-n-\delta,-\gamma}$ is the subspace of $\mathcal{D}'(\R^n\times\m T^m)$ consisting of the distributions which extend to give bounded linear functionals on $W^{p}_{2, \delta,\gamma}$, with the dual norm. 
		We study the kernel of $P^\ast$. Assume that $P^\ast u =0$ for $u \in W^{p'}_{0,-\delta-2,-\gamma}$. Since $P$ is symmetric, we have $Pu=0$. Then $u$ is smooth and, since $P$ has constant coefficients, $P \bar{u}=0$ and $P(u-\bar{u})=0$. Proposition \ref{lms1} implies then that $u-\bar{u}=0$. To prove that $\bar{u}=0$, we notice that, with estimate \eqref{estimp}, and since  $\bar{u}$ does not depend on the $\theta$ coordinates, $\bar{u}$ belongs to $W^{p'}_{2,-\delta}$. Moreover if we assume \eqref{condiso},  we have
		$$-\frac{n}{p'}=-n+\frac{n}{p}<-\delta-2,$$
		so by Proposition \ref{prpinj} we have that $\bar{u}=0$.
		
		In this case, we have that $P^*$ is injective, and therefore that $P$ is surjective. By this reasoning, both $P$ and $P^*$ are bijective.
		
		We consider $L$ equal to either $\Delta_g$ or $\overrightarrow{\Delta}_{g}$. To prove the bijectivity, we use the following continuity result from Theorem 5.2 in Gilbarg Trudinger \cite{giltru84}.
		
		\begin{thm}\label{thgt1}
			Let $S_0,S_1 : \q B \to \q V$ be bounded linear operator from a Banach space $\q B$ to a normed vector space $\q V$. Let $S_t=(1-t)S_0+tS_1$. We assume that there exists $C$ such that for all $t\in[0,1]$ we have
			$$\|x\|_{\q B}\leq C \|S_t x\|_{\q V}.$$
			Then $S_1$ is surjective if and only if $S_0$ is surjective.
		Moreover, the decomposition $S_t=(1-t)S_0+tS_1$ can be relaxed in $\|S_t-S_s\|_{\q B \to \q V} \leq C(t-s)$. 
		\end{thm}
		
		We are ready to conclude the proof of Theorem \ref{thinv}.
		\begin{proof}[Proof of Theorem \ref{thinv}] Let $g_{ij}-\zeta_{ij} \in W^p_{2,\sigma,\lambda}$. We work under the hypothesis \eqref{conditions}, \eqref{condinj} and \eqref{condiso}.
			Let us write, for $t=[0,1]$ : $g_t= (1-t)\zeta+ tg$, and consider $S_t$ equal to either $\Delta_{g_t}$ or $\overrightarrow{\Delta}_{g_t}$. Thanks to Corollary \ref{corinj} and Proposition \ref{prpinj} we have that $S_t$ is injective. Therefore \eqref{estimp} yields
			$$\|u\|_{W^p_{2,\delta,\gamma}}\lesssim \|S_t u\|_{W^p_{0,\delta+2,\gamma}}.$$
			Since $S_0$, which is either the Laplacian or the conformal Laplacian for the flat metric, is an isomorphism, it is the same for $S_1$, which is either $\Delta_g$ or $\overrightarrow{\Delta_g}$.
			This concludes the proof of Theorem \ref{thlin}.	
		\end{proof}
		
		\section{Solving the constraint equations}\label{sec_cont}
		
		\subsection{The barrier method}\label{secbar}
		\par The barrier method (or the sub- and super-solution method) was initially developed by Isenberg as a means of solving non-linear equations on a compact set in $\R^n$. Its proof can be extended to the asymptotically flat product manifold $(M,g)$ in the same way as for the asymptotically Euclidean manifold. 	{The following theorem is the equivalent of Theorem 1, Appendix B in \cite{ChoBruIseYor00}. We give the proof here for the sake of completeness, so that the reader is convinced that it is the same as the one for asymptotically Euclidean manifolds.} 
		\par Let $I=[l,m]$ be a bounded interval in $\R$. We define the operator $F:M\times I\to\R$ such that
		\begin{equation*}
			F(x,\theta,y)=\sum_{i} a_{P_i}(x,\theta)y^{P_i}.
		\end{equation*}
		Here, $P_i\in\R$ are a finite number of exponents and the coefficients $a_{P_i}\in W^p_{0,\delta+2,\gamma}$. We are interested in solving the equation $\Delta_g \varphi=F(x,\theta,\varphi)$ for $\varphi\in W^p_{2,\delta,\gamma}$. By a slight abuse of notation, we denote  $F(x,\theta,\varphi)$ the function $(x,\theta)\mapsto F(x,\theta,\varphi(x,\theta))$.
		\par A function $\varphi_{-}\in \mathcal{C}^2$ is called a subsolution of the equation $\Delta_g \varphi=F(x,\theta,\varphi)$ if 
		\begin{equation*}
			\Delta_g \varphi_{-}\leq F(x,\theta,\varphi_{-}).
		\end{equation*}
		Similarly, we call $\varphi_{+}\in \mathcal{C}^2$ a supersolution if 
		\begin{equation*}
			\Delta_g\varphi_{+}\geq F(x,\theta,\varphi_{+}).
		\end{equation*}
		\begin{thm} Let $g$ be a metric on $M$ such that $g_{ij}-\zeta_{ij}\in W^p_{2,\sigma,\lambda}(M)$. We assue the hypothesis of Theorem \ref{thexvac} for $(\sigma,\lambda,\delta,\gamma)$. {Let us assume that the $a_{P_i} \in W^p_{0,\delta+2,\gamma}$ are bounded.}
			Suppose that $\Delta_g\varphi=F(x,\theta,\varphi)$ admits a subsolution $\varphi_{-}$ and a supersolution $\varphi_{+}$ such that
			\begin{equation*}
				0<l\leq\varphi_{-}\leq \varphi_{+}\leq m,
			\end{equation*}
			and that 
			\begin{equation*}
				A_{-}:=\lim_{\infty}\varphi_{-}\leq A,\quad A_{+}:=\lim_{\infty}\varphi_{+}\geq A.
			\end{equation*}
			Then the equation admits a solution $\varphi$ that verifies
			\begin{equation*}
				\varphi_{-}\leq \varphi\leq\varphi_{+},\quad A-\varphi\in W^p_{s+2,\delta,\gamma}.
			\end{equation*}
			{In the case where $F(x,\theta,y)$ does not contain negative powers of $y$, the condition $l>0$ is not necessary.}
		\end{thm}
		\begin{proof}
			
			\par First, we note that if $\varphi $ takes values in $[l,m]$, then $F \in W^p_{0,\delta+2,\gamma}$.
			Let $k\in W^p_{0,\delta+2,\gamma}$ be a bounded positive function such that
			{\begin{equation*}
					k(x,\theta) \geq \sup_{l\leq y \leq m} \left|\partial_y F(x,\theta,y)\right|.
			\end{equation*}}
			\par We set $\varphi_1=A+u_1$, where $u_1\in W^p_{2,\delta,\gamma}$ is the unique solution of the linear equation
			\begin{equation*}
				\Delta_g u_1+ku_1=F(x,\theta,\varphi_{-})+k(\varphi_{-}-A),
			\end{equation*}
			as seen in Theorem \ref{thinv}. The function $\varphi_1$ can easily be shown to take values between $\varphi_{-}$ and $\varphi_{+}$. The first bound follows directly from the maximum principle of Lemma \ref{positivity cor}, given that
			\begin{equation*}
				\Delta_g(\varphi_{1}-\varphi_{-})+k(\varphi_{1}-\varphi_{-})\geq 0
			\end{equation*}
			and $\varphi_{1}-\varphi_{0}$ tends to $A-A_{-}\geq 0$ at infinity.
			\par On the other hand,
			\begin{equation*}
				\Delta_g (\varphi_{+}-\varphi_1)+k(\varphi_{+}-\varphi_1){\geq F(x,\theta,\varphi_{+})-F(x,\theta,\varphi_{-})+k(\varphi_{+}-\varphi_{-})\geq 0}
			\end{equation*}
			and $\varphi_{+}-\varphi_1$ tends to $A_{+}-A\geq 0$ at infinity, then by the same maximum principle we obtain the desired result.
			\par We proceed by recurrence and define $\varphi_i=A+u_i$, with
			\begin{equation*}
				\Delta_g u_i + k u_i=F(x,\theta,\varphi_{i-1})+ku_{i-1}.
			\end{equation*}
			For the sake of this argument, we may denote $\varphi_{-}=u_{0}$. In order to show that the sequence is increasing, we may assume that $u_{i-1}\geq u_{i-2}$  (the initial step $u_1\geq u_0$ has been shown above). Since
			\begin{equation}
				\Delta_g u_{i-1} + k u_{i-1}=F(x,\theta,\varphi_{i-2})+ku_{i-2}.
			\end{equation}
			we see that $(\Delta+k)(u_i- u_{i-1})\geq 0$. Since $u_i-u_{i-1}$ tends to $0$ at infinity, the maximum principle tells us that $u_i\geq u_{i-1}$, and thus $\varphi_i\geq \varphi_{i-1}$.
			\par As before, we can easily see that $\varphi_{-}\leq \varphi_i\leq \varphi_{+}$. {Consequently, for all $x$, $\varphi_i(x)$ has a limit that we note $\varphi(x)$.}
			\par We would like to show that $\varphi-A\in W^p_{2,\delta,\gamma}$ and that it is a solution of our equation. We use Theorem \ref{thinv} to see that
			\begin{equation*}
				\|u_{i+1}\|_{W^p_{2,\delta,\gamma}}\leq C\|F(x,\theta,\varphi_i)+ku_i\|_{W^p_{0,\delta+2,\gamma}},
			\end{equation*}
			which implies that $(u_i)_{i\in\m N}$ are uniformly bounded in $W^p_{2,\delta,\gamma}$ (thanks to the uniform born $l\leq \varphi_i\leq m$). We can thus extract a sequence that converges in $W^p_{1,\delta',\gamma'}$ norm to a function in $u\in {W^p_{2,\delta,\gamma} }$, where $-\frac{n}{p}<\delta'<\delta$ and $0<\gamma'<\gamma$.
			\par {We recall the product estimate
				$$W^p_{1,\delta',\gamma'}\times W^p_{0,\delta+2,\gamma} \to W^p_{0,\delta+2,\gamma}$$
				which holds if $\frac{n+m}{p}<1$, $\gamma \leq 2+\delta+\gamma'$ and $\delta+2\leq \gamma+\gamma'$. Since $\delta+2>0$, these conditions can be fulfied with $\gamma'$ close to $\gamma$.
				Thanks to this product estimate, we can show that $F(x,\theta,\varphi_i)$ converges toward $F(x,\theta,\varphi)$ in $W^p_{0, \delta+2,\gamma}$, 
				which implies that $\varphi$ is a solution to
				$$\Delta_g \varphi = F(x,\theta,\varphi).$$}
		\end{proof}
		
		\subsection{The Hamiltonian equation}\label{secham}
		\begin{thm}\label{thham}
			Let $g_{ij}-\zeta_{ij} \in W^p_{2,\sigma,\lambda}(M)$ and let $a,b\geq 0$, $a,b,h\in W^p_{0,\delta+2,\gamma}$, bounded. We assume the hypothesis of Theorem \ref{thexvac} for $(\sigma,\lambda,\delta,\gamma)$. The equation
			\begin{equation*}
				\Delta_g \varphi + h \varphi = -b\varphi^{\qq-1}+\frac{a}{\varphi^{\qq+1}}
			\end{equation*}
			has a solution $\varphi=A+u$, $u\in W^p_{2,\delta,\gamma}$, $A>0$, $\varphi>0$.
		\end{thm}
		\begin{proof}
			\par We identify a subsolution $\varphi_{-}\leq A$ and a supersolution $\varphi_{+}\geq A$ to the equation, where $\varphi_{-}\in W^p_{2,\delta,\gamma}$ and $\varphi_{+}\in W^p_{2,\delta,\gamma}$ solve
			\begin{equation}\label{subsolution eq}
				\Delta_g \varphi_{-}+h\varphi_{-}+b\varphi_{-}^{\qq-1}=0
			\end{equation}
			and
			\begin{equation}\label{supersolution eq}
				\Delta_g\varphi_{+}-\frac{a}{\varphi_+^{\qq+1}}=0,
			\end{equation}
			respectively. In order to solve both \eqref{subsolution eq} and \eqref{supersolution eq}, we once more apply the barrier method, and are thus tasked with finding subsolutions and supersolutions for each of them.
			\par \textit{The subsolution $\varphi_{-}$ solving \eqref{subsolution eq}}. The subsolution of equation \eqref{subsolution eq} is chosen to be as the constant  $\psi^{-}=0$. The supersolution of equation \eqref{subsolution eq} is a constant $\psi^{+}>A$. Therefore, we can always find a non-negative solution $\varphi_{-}$. Finally, the positivity of $\varphi_{-}$ follows from Alecksandrov's uniqueness theorem (Theorem 1.7, \cite{Can79}, a special case of Theorem A, \cite{smi62}). Consider the linear equation 
			\begin{equation*}
				\Delta_g \varphi+(h+b\varphi_{-}^{\qq-2})\varphi=0,
			\end{equation*}
			which clearly accepts both $\varphi_{-}$ and $0$ as solutions. If there exists $p\in M$ such that $\varphi_{-}(p)= 0$, then this is a minimum and $\nabla\varphi_{-}=0$. By Alecksandrov's uniqueness theorem,  $\varphi_{-}\equiv 0$, which contradicts the fact that $\varphi_{\lambda_0}$ tends to $A>0$ at infinity.
			\par \textit{The supersolution $\varphi_+$ which solves \eqref{supersolution eq}.} We fix the subsolution $\psi^{-}=A$ to \eqref{supersolution eq}. We find a supersolution $\psi^{+}$ as the limit of a sequence constructed iteratively as follows: let $\psi_0=A$, and $\psi_i=A+u_i$ 
			\begin{equation*}
				\Delta_g u_i =\frac{a}{\psi_{i-1}^{\qq+1}}.
			\end{equation*}
			We prove iteratively that $u_i\geq 0$.  We get that
			\begin{equation*}
				\|u_i\|_{W^p_{2,\delta,\gamma}}\leq C\left\| \frac{a}{(u_{i-1}+A)^{\qq+1}}\right\|_{W^p_{0,\delta+2,\gamma}}\leq C\|a\|_{W^p_{0,\delta+2,\gamma}}.
			\end{equation*}
			We can thus extract a subsequence that converges in $W^p_{1,\delta',\gamma'}$ to $\varphi_{+}-A\in W^p_{2,\delta,\gamma}$, with $\delta'<\delta, \gamma'<\gamma$. Then $\varphi_{+}$ is then a solution of \eqref{supersolution eq}.
		\end{proof}

		\subsection{The coupled system }\label{secfix}
		\par We are ready to provide the proof for Theorem \ref{thexvac}, which treats the problem of the existence of initial data for the vacuum case. We prove it with an iteration scheme, using a compactness argument.
		\begin{proof}[Proof of Theorem \ref{thexvac}.]
			In order to prove the existence of a solution of the coupled system, we want to construct a convergent sequence $\varphi_i=A+u_i$, defined by iteration. 
			We start with $\varphi_0=0$. If $\varphi_{i-1}$ is a function, bounded by a constant $M$, we use Theorem \ref{thinv} to define $W_i \in W^p_{2,\delta,\gamma}$ to be the unique solution of the linear equation, 
			\begin{equation*}
				\Delta_g W_{i}=\frac{d-1}{d}\varphi_{i}^{\qq}d\tau.
			\end{equation*}
			It satisfies the estimate
			\begin{equation}\label{estWi}
				\|W_i\|_{W^p_{2,\delta,\gamma}}\leq CM\|d\tau\|_{W^p_{0,\delta+2,\gamma}}.
			\end{equation}
			Then we define the scalar field $\varphi_i$ to be the solution, given by Theorem \ref{thham}, of
			\begin{equation*}
				\frac{4(d-1)}{d-2}\Delta_g \varphi_i + R_g\varphi_i= -\frac{d-1}{d}\tau^2\varphi_i^{\qq-1}+\frac{a(W_{i})}{\varphi_i^{\qq+1}},
			\end{equation*}
			where
			\begin{equation*}
				a(W_{i-1})=\vert\mathcal{L}_g W_{i-1}+U\vert^2.
			\end{equation*}
			The solution $\varphi_i$ can be written
			$\varphi_i=A+u_i$, with $u_i \in W^2_{p,\delta,\gamma}$. To show the compactness of the sequence, it is sufficient to show that $u_i$ is uniformly bounded in $W^2_{p,\delta,\gamma}$.
			Solving the scalar equation requires the existence of sub and super solutions at each iteration. In fact, the most straightforward way to go about it is to use the same sub and super solution at each iteration, such that $\varphi_{-}\leq \varphi_i\leq \varphi_{+}$. 
			It will ensure that the $\varphi_i$ are all uniformly bounded by the same constant $M=\sup \varphi_{+}$.
			The subsolution $\varphi_{-}$ solves the same equation \eqref{subsolution eq} as before (it does not depend on $W_i$ !).  As for the supersolution, we ask that it solves
			\begin{equation*}
				\Delta_g \varphi_{+}=\frac{\mathcal{A}}{\varphi_{+}^{\qq+1}},
			\end{equation*}
			where $\mathcal{A}\in W^p_{0,\delta+2,\gamma}$, fixed, is to be chosen sufficiently large such that for all $i$, $a(\varphi_i)\leq \mathcal{A}$, in order to ensure that $\varphi_+$ stays a supersolution each iteration. Therefore, by default, 
			\begin{equation*}
				\mathcal{A}>{\vert U\vert^2}.
			\end{equation*}
			Thanks to the embeding $W^p_{1,\delta+1,\gamma} \subset C^0_\rho$ with $\rho<\min(\gamma, \delta+1+\frac{n}{p})$, we can write
			$$|\q L_g W| \leq C\frac{\|W_i\|_{W^p_{2,\delta,\gamma}}}{\langle x \rangle^{2\rho}} \leq \frac{CM\ep}{\langle x \rangle^{2\rho}} ,$$
			where we have used \eqref{estWi} the hypothesis $\|d\tau\|_{W^p_{0,\delta+2,\gamma}}\leq \ep$. We choose $\ep$ small enough in order to have $C\ep M \leq 1$.
			Then we can write
			$$a(W_{i-1}) \leq 2 |U|^2+\frac{2}{\langle x \rangle^{2\rho}}  = \q A.$$
			In this way, we have defined a fixed function $\q A$, which belongs to $W^p_{0,\delta+2,\gamma}$ if $2\rho>\delta+2+\frac{n}{p}$. This condition is satisfied under the conditions $\delta+\frac{n}{p}>0$ and $2\gamma >\delta+2+\frac{n}{p}$.
			We have consequently constructed a uniform supersolution $\varphi_+$.
			
			\par Now that we have proved the functions $\varphi_i$ are uniformly bounded, from bellow and from above, we have a uniform bound for $	\|W_i\|_{W^p_{2,\delta,\gamma}}$, and the elliptic estimate yield a uniform bound for $\|u_i\|_{W^p_{2,\delta,\gamma}}$.
			By compactness, we can extract a subsequence $\varphi_{f(i)}$ which converges uniformly in $C^0$ to some $\varphi$, and such that $\varphi_{f(i)}-A$ converges weakly in $W^p_{2,\delta,\gamma}$. Then $W_{f(i)}$ converges strongly in $W^2_{2,\delta,\gamma}$ to some $W$, and the limit $(\varphi,W)$ is a solution to the constraint equations. 
			
		\end{proof} 
		\bibliographystyle{alpha} 
		\bibliography{refs} 

\begin{thebibliography}{ABWY20}

\bibitem[ABWY20]{AndBluWyaYau20}
Lars Andersson, Pieter Blue, Zoe Wyatt, and Shing-Tung Yau.
\newblock {Global stability of spacetimes with supersymmetric
  compactifications}.
\newblock 6 2020.

\bibitem[Bar86]{Bar86}
Robert Bartnik.
\newblock The mass of an asymptotically flat manifold.
\newblock {\em Communications on Pure and Applied Mathematics}, 39(5):661--693,
  1986.

\bibitem[BFK19]{BraFajKro19}
Volker Branding, David Fajman, and Klaus Kr\"oncke.
\newblock {Stable Cosmological Kaluza\textendash{}Klein Spacetimes}.
\newblock {\em Commun. Math. Phys.}, 368(3):1087--1120, 2019.

\bibitem[BL87]{kk}
D.~Bailin and A.~Love.
\newblock {Kaluza-Klein theories}.
\newblock {\em Rep. Prog. Phys.}, 1987.

\bibitem[Can79]{Can79}
Murray Cantor.
\newblock A necessary and sufficient condition for york data to specify an
  asymptotically flat spacetime.
\newblock {\em Journal of Mathematical Physics}, 20(8):1741--1744, 1979.

\bibitem[CB09]{livrecb}
Yvonne Choquet-Bruhat.
\newblock {\em General relativity and the {E}instein equations}.
\newblock Oxford Mathematical Monographs. Oxford University Press, Oxford,
  2009.

\bibitem[CbC81]{cb81}
Y.~Choquet-bruhat and D.~Christodoulou.
\newblock {Elliptic systems in H s d spaces on manifolds which are euclidean at
  infinity}.
\newblock {\em Acta Math.}, 146:129--150, 1981.

\bibitem[CBIY00]{ChoBruIseYor00}
Yvonne Choquet-Bruhat, James Isenberg, and James~W. York, Jr.
\newblock {Einstein constraints on asymptotically Euclidean manifolds}.
\newblock {\em Phys. Rev.}, D61:084034, 2000.

\bibitem[CBY80]{ChoYor80}
Yvonne Choquet-Bruhat and James~W. York, Jr.
\newblock The {C}auchy problem.
\newblock In {\em General relativity and gravitation, {V}ol. 1}, pages 99--172.
  Plenum, New York-London, 1980.

\bibitem[CHSW85]{CanHorStrWit85}
P.~Candelas, Gary~T. Horowitz, Andrew Strominger, and Edward Witten.
\newblock {Vacuum Configurations for Superstrings}.
\newblock {\em Nucl. Phys. B}, 258:46--74, 1985.

\bibitem[CM14]{chrmaz14}
Piotr~T. Chruściel and Rafe Mazzeo.
\newblock Initial data sets with ends of cylindrical type: I. the lichnerowicz
  equation, 2014.

\bibitem[CMP13]{chrmazpoc13}
Piotr~T. Chruściel, Rafe Mazzeo, and Samuel Pocchiola.
\newblock {Initial data sets with ends of cylindrical type: II. The vector
  constraint equation}.
\newblock {\em Advances in Theoretical and Mathematical Physics}, 17(4):829 --
  865, 2013.

\bibitem[COM81]{christodoulou_omurchadha}
D.~Christodoulou and N.~\'{O}~Murchadha.
\newblock The boost problem in general relativity.
\newblock {\em Comm. Math. Phys.}, 80(2):271--300, 1981.

\bibitem[Dai04]{Dai04}
Xianzhe Dai.
\newblock A positive mass theorem for spaces with asymptotic susy
  compactification.
\newblock {\em Communications in Mathematical Physics}, 244(2):335–345, Jan
  2004.

\bibitem[Duf94]{duff}
M.J. Duff.
\newblock {Kaluza-Klein Theory in Perspective}.
\newblock 1994.

\bibitem[GT84]{giltru84}
D.~Gilbarg and N.~Trudinger.
\newblock {\em Elliptic Partial Differential Equations of Second Order}.
\newblock Grundlehren der mathematischen Wissenschaften. Springer Berlin
  Heidelberg, 1984.

\bibitem[Kal18]{Kal21}
Th. Kaluza.
\newblock {Zum Unit\"atsproblem der Physik}.
\newblock {\em Int. J. Mod. Phys. D}, 27(14):1870001, 2018.

\bibitem[Ker68]{kern}
Richard Kerner.
\newblock {Generalization of the Kaluza-Klein theory for an arbitrary
  non-abelian gauge group}.
\newblock {\em Annales Henri Poincaré}, 1968.

\bibitem[{Kle}26]{Kle26}
Oskar {Klein}.
\newblock {Quantentheorie und f{\"u}nfdimensionale Relativit{\"a}tstheorie}.
\newblock {\em Zeitschrift fur Physik}, 37(12):895--906, December 1926.

\bibitem[Lic44]{Lic44}
Andr\'{e} Lichnerowicz.
\newblock L'int\'{e}gration des \'{e}quations de la gravitation relativiste et
  le probl\`eme des {$n$} corps.
\newblock {\em J. Math. Pures Appl. (9)}, 23:37--63, 1944.

\bibitem[Max04]{Max04}
David Maxwell.
\newblock Solutions of the einstein constraint equations with apparent horizon
  boundaries.
\newblock {\em Communications in Mathematical Physics}, 253(3):561–583, Nov
  2004.

\bibitem[SMP93]{Ste93}
E.M. Stein, T.S. Murphy, and Princeton~University Press.
\newblock {\em Harmonic Analysis: Real-variable Methods, Orthogonality, and
  Oscillatory Integrals}.
\newblock Monographs in harmonic analysis. Princeton University Press, 1993.

\bibitem[SS62]{smi62}
I.U.M. Smirnov and J.M. Smirnov.
\newblock {\em 9 Papers on Topology, Lie Groups, and Differential Equations}.
\newblock American Mathematical Society: Translations : 2. series. American
  Mathematical Society, 1962.

\bibitem[Wya18]{Wya18}
Zoe Wyatt.
\newblock The weak null condition and kaluza–klein spacetimes.
\newblock {\em Journal of Hyperbolic Differential Equations}, 15(02):219–258,
  Jun 2018.

\end{thebibliography}
	\end{document}